\documentclass{amsart}
\usepackage{amssymb}
\usepackage{amscd}
\usepackage[all]{xy}
\usepackage{subfigure}
\usepackage{graphicx}
\usepackage{color}
\usepackage{hyperref}

\newtheorem{thm}{Theorem}[section]
\newtheorem{lem}[thm]{Lemma}
\newtheorem{prop}[thm]{Proposition}
\newtheorem{cor}[thm]{Corollary}

\theoremstyle{definition}
\newtheorem{defn}[thm]{Definition}
\newtheorem{rmk}[thm]{Remark}

\def\Aut{\operatorname{Aut}}
\def\Out{\operatorname{Out}}
\def\Inn{\operatorname{Inn}}

\def\coker{\operatorname{coker}}
\def\tv{\operatorname{tv}}

\begin{document}
\title{Automorphisms of braid groups on orientable surfaces}
\author{Byung Hee An}
\thanks{This work was supported by IBS-R003-D1.}
\address{Center for Geometry and Physics, Institute for Basic Science (IBS), Pohang, Republic of Korea 37673}
\email{anbyhee@ibs.re.kr}
\begin{abstract}
In this paper we compute the automorphism groups $\Aut(\mathbf{P}_n(\Sigma))$ and $\Aut(\mathbf{B}_n(\Sigma))$ of braid groups $\mathbf{P}_n(\Sigma)$ and $\mathbf{B}_n(\Sigma)$ on every orientable surface $\Sigma$, which are isomorphic to group extensions of the extended mapping class group $\mathcal{M}^*_n(\Sigma)$ by the transvection subgroup except for a few cases.

We also prove that $\mathbf{P}_n(\Sigma)$ is always a characteristic subgroup of $\mathbf{B}_n(\Sigma)$ unless $\Sigma$ is a twice-punctured sphere and $n=2$.
\end{abstract}
\maketitle

\section{Introduction}
Let $\Sigma$ be a surface of finite type without boundary. Unless mentioned otherwise, we always assume that $\Sigma$ is orientable and homeomorphic to $\Sigma_{g,p}=\Sigma_g\setminus\mathbf{p}$, where $\Sigma_g$ denotes a closed surface of genus $g$ and $\mathbf{p}=\{p_1,\dots,p_p\}$ is a set of {\em punctures}.
The {\em ordered configuration space} $F_n(\Sigma)$ is defined to be the space of ordered $n$-tuples of distinct points in $\Sigma$.
The symmetric group $\mathbf{S}_n$ acts on $F_{n}(\Sigma)$ by permuting coordinates, and the orbit space $F_{n}(\Sigma)/\mathbf{S}_n$ is called the {\em unordered configuration space} of $\Sigma$.
The fundamental groups of $F_{n}(\Sigma)$ and $F_{n}(\Sigma)/\mathbf{S}_n$ are called the {\em pure} and {\em (full) $n$-braid groups on $\Sigma$} denoted by $\mathbf{P}_{n}(\Sigma)$ and $\mathbf{B}_{n}(\Sigma)$, respectively.

If we fix a basepoint $\mathbf{z}=(z_1,\dots,z_n)$ for $F_n(\Sigma)$, there is a short exact sequence
\begin{equation*}
1\longrightarrow\mathbf{P}_{n}(\Sigma)\longrightarrow\mathbf{B}_{n}(\Sigma)\stackrel{\rho}{\longrightarrow}\mathbf{S}_n\longrightarrow 1,
\end{equation*}
where the homomorphism $\rho$ is called the {\em induced permutation} and $\rho(\beta)$ for $\beta\in\mathbf{B}_{n}(\Sigma)$ is precisely a permutation of $\mathbf{z}$.
The main objects of study in this paper are the automorphism groups $\Aut(\mathbf{P}_n(\Sigma))$ and $\Aut(\mathbf{B}_n(\Sigma))$.

It is worth remarking that $\mathbf{B}_n(\Sigma_{g,p})$ can be regarded as a subgroup of $\mathbf{B}_{n+p}(\Sigma_g)$. Roughly speaking, it is a finite index subgroup consisting of braids in $\mathbf{B}_{n+p}(\Sigma_g)$ whose first $n$ strands end at the first $n$ marked points. 

The {\em extended mapping class groups} 
$\mathcal{M}^*_{n}(\Sigma)$ and $\mathcal{M}^*(\Sigma)$ are defined to be the groups of
isotopy classes of (possibly orientation-reversing) homeomorphisms of $(\Sigma,\mathbf{z})$ and $\Sigma$, respectively. 
These groups can be regarded as subgroups of the mapping class groups $\mathcal{M}^*_{n+p}(\Sigma_g)$ on $\Sigma_g$ with $(n+p)$-marked points $\mathbf{z}\cup\mathbf{p}$.

There is a short exact sequence due to Birman as follows.
\begin{equation}
\label{eqn:braidMCG}
1\longrightarrow\mathbf{B}_{n}(\Sigma)/Z(\mathbf{B}_{n}(\Sigma))\xrightarrow{Push}\mathcal{M}^*_{n}(\Sigma)
\longrightarrow\mathcal{M}^*(\Sigma)\longrightarrow 1,
\end{equation}
where $Z(\mathbf{B}_n(\Sigma))$ denotes the center.

Evidently, mapping class groups are related with the automorphism groups of braid groups since the group on the left is the inner automorphism group of $\mathbf{B}_{n}(\Sigma)$, a subgroup of the automorphism group $\Aut(\mathbf{B}_n(\Sigma))$.

\subsection{Previous results}

Let us briefly review known results relating automorphisms of braid groups to mapping class groups. The first one we introduce is a generalization of Dehn-Neilsen's theorem.
\begin{thm}\cite[Theorem~3.3.11]{ZVC}\label{thm:peripheral}
Let $\phi\in\Aut(\pi_1(\Sigma))$ be an automorphism. There exists a homeomorphism $f\in\mathcal{M}^*_{1}(\Sigma)$ realizing $\phi$ if and only if $\phi$ preserves the peripheral structure.

In particular, if $\Sigma$ is closed, then such an $f$ always exists.
\end{thm}

Hence this theorem gives a complete criterion for a given automorphism of $\mathbf{B}_{1}(\Sigma)=\pi_1(\Sigma)$ to be realized by a homeomorphism. 

For $S^2$ with $n\le3$, the braid group $\mathbf{B}_n(S^2)$ is finite, and so is $\Aut(\mathbf{B}_n(S^2))$.
Bellingeri in \cite[Theorem~6.2]{Bel2} showed that for $n\ge 4$, $\Out(\mathbf{B}_n(S^2))\simeq \mathbb{Z}_2\times\mathbb{Z}_2$
which is isomorphic to $\tv(\mathbf{B}_n(S^2))\times\mathcal{M}^*(S^2)$, where $\tv(G)$ is the {\em transvection subgroup} of $G$ defined as the kernel of $\Aut(G)\to\Aut(G/Z(G))$.
Indeed, $\tv(\mathbf{B}_n(S^2))\simeq\mathbb{Z}_2$ is generated by $\sigma_i\mapsto\sigma_i\Delta^2$ for each Artin generator $\sigma_i$ and the central element $\Delta^2$ in $\mathbf{B}_n(S^2)$.
Therefore, this result together with the Birman exact sequence (\ref{eqn:braidMCG}) implies that for $\Sigma=S^2$ and $n\ge 4$,
\begin{align}\label{eq:tv}
\Aut(\mathbf{B}_n(\Sigma))\simeq \tv(\mathbf{B}_n(\Sigma))\rtimes\mathcal{M}^*_n(\Sigma).
\end{align}
Moreover, this can be obtained from the following theorem due to Ivanov and Korkmaz as well.
\begin{thm}\cite[Theorem~2]{Iv2}\cite[Theorem~1]{K}\label{thm:AutMCG} Suppose that $n\ge 5$.
Then any isomorphism between two finite index subgroups of $\mathcal{M}^*_{n}(S^2)$ is the restriction of an inner automorphism of $\mathcal{M}^*_{n}(S^2)$.
\end{thm}

For $g\ge 1$, Irmak, Ivanov and McCarthy proved in \cite[Theorem~1, Theorem~2]{IIM} that
by using the Nielsen-Thurston's classifications of mapping classes, 
\[
\Aut(\mathbf{P}_n(\Sigma_1)/Z(\mathbf{P}_n(\Sigma_1)))\simeq\mathcal{M}^*_n(\Sigma_1), \quad 
\Aut(\mathbf{B}_n(\Sigma_g))\simeq\Aut(\mathbf{P}_n(\Sigma_g))\simeq\mathcal{M}^*_n(\Sigma_g),
\]
for $g\ge 2$ and $n\ge 3$. 
In 2011, Kida and Yamagata in \cite[Theorem~1.1]{KY} proved a result similar to Theorem~\ref{thm:AutMCG} for pure braid groups with $g,n\ge 2$, and one of its immediate consequences is that $\Aut(\mathbf{P}_2(\Sigma_g))\simeq\mathcal{M}^*_2(\Sigma_g)$ for $g\ge 2$.

Zhang in \cite{Zh1, Zh2} showed that (\ref{eq:tv}) holds for all closed surfaces including non-orientable cases with $n\ge 2$. However some of his arguments are incorrect. See the comment before Lemma~\ref{lem:transvection}. 

On the other hand, for non-closed cases only a few results are known.
In \cite[Theorem~19, Theorem~20]{DG}, Dyer and Grossman computed the automorphism group of the classical braid group
$\mathbf{B}_n(\mathbb{R}^2)$,
which is
$\Aut(\mathbf{B}_{n}(\mathbb{R}^2))\simeq\mathcal{M}^*_{n}(\mathbb{R}^2)$ for all $n\ge 2$. Bell and Margalit in \cite[Theorem~7, Theorem~8]{BM} and Charney and Crisp in \cite[Theorem~1~(ii), Proposition~5]{CC} computed automorphism groups of several Artin groups including $A(B_n)\simeq\mathbf{B}_n(\Sigma_{0,2})$, $A(\tilde C_n)\simeq\mathbf{B}_n(\Sigma_{0,3})$, and the classical pure braid group $P(A_{n-1})\simeq\mathbf{P}_n(\mathbb{R}^2)$.
Their results are based on Theorem~\ref{thm:AutMCG} and can be formulated as follows.
The isomorphism (\ref{eq:tv}) holds for $\Sigma_{0,2}$ with $n\ge 3$ and for $\Sigma_{0,3}$ with $n\ge 2$, and
\begin{align*}
\Aut(\mathbf{P}_{n}(\mathbb{R}^2))\simeq\tv(\mathbf{P}_{n}(\mathbb{R}^2))\rtimes\mathcal{M}^*_{n+1}(S^2).
\end{align*}

Notice that since $\mathbf{P}_{n}(\Sigma_{1,1})\simeq\mathbf{P}_n(\Sigma_1)/Z(\mathbf{P}_n(\Sigma_1))$ \cite[Lemma~17]{BGG}, the once-punctured torus case essentially comes from the closed case, which is known for $n\ge 3$ by \cite[Theorem~1]{IIM}, and therefore all known cases are when Euler characteristic $\chi(\Sigma)\ge -1$.

Motivated by this fact, for a not necessarily closed surface $\Sigma$, we say that $\Sigma$ is {\em generic} if $\chi(\Sigma)<-1$, and call all non-generic missing cases for $\Aut(\mathbf{B}_n(\Sigma))$ {\em exceptional}. They are as follows: (i) the torus $\Sigma_1$ with $n\ge2$;
(ii) the once-punctured torus $\Sigma_{1,1}$ with $n\ge2$; and (iii) the twice-punctured sphere $\Sigma_{0,2}$ with $n=2$. 

We remark that non-generic cases for $\Aut(\mathbf{P}_n(\Sigma))$ are either obvious, reduced to generic cases, or lying in the scope of Theorem~\ref{thm:AutMCG}. See Remark~\ref{rmk:exceptionalPn}.

\subsection{Results}
The main result of this paper is as follows.

\begin{thm}\label{thm:AutPn}
Let $\Sigma$ be a generic surface and $n\ge 2$. Then
\[
\Aut(\mathbf{P}_{n}(\Sigma))\simeq\mathcal{M}^*_{n}(\Sigma),\quad
\Out(\mathbf{P}_{n}(\Sigma))\simeq\mathbf{S}_n\times\mathcal{M}^*(\Sigma),
\]
and
\[
\Aut(\mathbf{B}_{n}(\Sigma))\simeq\mathcal{M}^*_{n}(\Sigma),\quad
\Out(\mathbf{B}_{n}(\Sigma))\simeq\mathcal{M}^*(\Sigma).
\]
\end{thm}

This overlaps the previously known cases (when $p=0$ and $n\ge 3$). Otherwise, this result is new and covers all but only a few exceptions, which are computed separately.
Then together with all known and exceptional cases, it can be summarized as follows.

\begin{thm}\label{thm:AutBn}
Let $\Sigma$ be an orientable surface of finite type without boundary, and let $Z=Z(\mathcal{M}^*_n(\Sigma))$.
Then for any $n\ge 2$,
\[
\Aut(\mathbf{B}_n(\Sigma))\simeq \tv(\mathbf{B}_n(\Sigma))\rtimes\mathcal{M}^*_n(\Sigma)/Z\rtimes G_n(\Sigma)
\]
and
\[
\Out(\mathbf{B}_n(\Sigma))\simeq \tv(\mathbf{B}_n(\Sigma))\rtimes\mathcal{M}^*(\Sigma)/Z\rtimes G_n(\Sigma),
\]
where 
\[
G_n(\Sigma)=\begin{cases}
\mathbb{Z}_2 & (\Sigma,n)=(A,2);\\
1&otherwise.
\end{cases}
\]
\end{thm}

As an application, we prove the following theorem, which extends the result of Bellingeri and Ivanov in \cite{Bel2, Iv1}.

\begin{thm}\label{thm:Characteristic}
Let $\Sigma$ be a surface of finite type without boundary, possibly non-orientable. Then the pure braid group $\mathbf{P}_n(\Sigma)$ is a characteristic subgroup of the braid group $\mathbf{B}_n(\Sigma)$ except $\mathbf{P}_2(\Sigma_{0,2})$ is not characteristic in $\mathbf{B}_2(\Sigma_{0,2})$.
\end{thm}

The rest of this paper is organized as follows. In Section~2, we introduce known facts of braid groups and mapping class groups, especially about the centers $Z(\mathbf{B}_n(\Sigma))$ and special normal subgroups $\mathbf{U}_i$ of the pure braid groups which are the kernels of the map $q_{i*}:\mathbf{P}_n(\Sigma)\to\mathbf{P}_{n-1}(\Sigma)$ forgetting the $i$-th strand. 
These two subgroups are very useful in studying not only braid groups but also their automorphism groups. Moreover, we consider natural maps from the mapping class groups to automorphism groups of braid groups, whose images can be thought of as geometric subgroups of $\Aut(\mathbf{B}_n(\Sigma))$ in the sense that they are obviously realizable by homeomorphisms.
On the other hand, the algebraic subgroups can be thought of as coming from the transvection subgroups defined above.

In Section~3, we compute the transvection subgroups for both the pure and full braid groups, and compute automorphism groups for all but the exceptional cases by assuming Theorem~\ref{thm:Characteristic} and Proposition~\ref{prop:AutPn} which is the first half of Theorem~\ref{thm:AutPn}.
In Section~4, by using combinatorial group theory, we find necessary and sufficient conditions for a given automorphism to be realizable, and prove Proposition~\ref{prop:AutPn} by showing that our conditions always occur in the generic case. To this end, we use an analog of the argument of Irmak, Ivanov and McCarthy in \cite{IIM}, and analyze the rank of the center of the centralizer for a given braid.
This argument is then used to prove Theorem~\ref{thm:Characteristic} in Section~5. Finally, all exceptional cases are treated in Section~6.

\begin{rmk}[boundary versus punctures]
Suppose that $\Sigma$ is a {\em compact} surface of genus $g$ with $b\ge1$ boundary components. The mapping class group $\mathcal{M}_n(\Sigma,\partial\Sigma)$ is defined as the group consisting of isotopy classes of orientation-preserving homeomorphisms on $(\Sigma, \mathbf{z})$ fixing $\partial\Sigma$ {\em pointwise}. 
Then 
Dehn twists along boundary components act trivially on braid groups, and moreover they generate the center $Z(\mathcal{M}_n(\Sigma,\partial\Sigma))$ unless $\Sigma$ is the closed disc $D^2$ and $n=2$ \cite[Theorem~5.6]{PR2}.

Moreover, $\mathcal{M}_n(\Sigma,\partial\Sigma)/Z(\mathcal{M}_n(\Sigma,\partial\Sigma))$ is isomorphic to a subgroup of $\mathcal{M}^*_n(\mathring{\Sigma})$ of index $2\cdot b!$,
where $\mathring{\Sigma}$ is the interior of $\Sigma$ and homeomorphic to $\Sigma_{g,b}$.
Here the absence of orientation-reversing maps contributes to the index as much as 2 and the index $b!$ arises since a mapping class on $\Sigma_{g,b}$ may permute the $b$ punctures but a mapping class in $\mathcal{M}_n(\Sigma,\partial\Sigma)$ never permutes boundary components. 
However, one can easily find a mapping class in $\mathcal{M}_n^*(\Sigma_{g,b})\setminus \left(\mathcal{M}_n(\Sigma,\partial\Sigma)/Z(\mathcal{M}_n(\Sigma,\partial\Sigma))\right)$ such that the induced automorphism on the braid group is nontrivial. 
Therefore $\mathcal{M}_n(\Sigma,\partial\Sigma)$ is insufficient for describing the whole automorphism group.

On the contrary, $\mathbf{B}_n(\mathbb{R}^2)$ and $\mathbf{B}_n(A)$ for $A=(0,1)\times S^1$ can be identified with $\mathcal{M}_n(D^2,\partial D^2)$ and a subgroup of $\mathcal{M}_n(\bar A,\partial \bar A)$ for $\bar A=[0,1]\times S^1$, respectively.
Then as mentioned before, $\Aut(\mathbf{B}_n(\mathbb{R}^2))\simeq\mathcal{M}^*_n(\mathbb{R}^2)$ and $\Aut(\mathbf{B}_n(A))/\tv(\mathbf{B}_n(A))\simeq\mathcal{M}^*_n(A)$.
However, these are different from $\mathcal{M}_n(D^2,\partial D^2)$, $\mathcal{M}_n(D^2,\partial D^2)/Z(\mathcal{M}_n(D^2,\partial D^2))$, $\mathcal{M}_n(\bar A,\partial \bar A)$ and $\mathcal{M}_n(\bar A,\partial \bar A)/Z(\mathcal{M}_n(\bar A,\partial \bar A))$.

This is the reason why we consider surfaces without boundary rather than compact surfaces, and we suggest that the reader {\em not} identify the braid groups $\mathbf{B}_n(\mathbb{R}^2)$ and $\mathbf{B}_n(A)$ with (subgroups of) mapping class groups as above.
\end{rmk}

\section{Braid groups and mapping class groups}\label{sec:braidAndMCG}

For convenience, we use $G_{/Z}$ to denote the quotient $G/Z(G)$ of $G$ by its center,
and use the shorthand notations
$S^2$, $\mathbb{R}^2$, $A$, and $T$ for $\Sigma_0$, $\Sigma_1, \Sigma_{0,2}$ and $\Sigma_1$, respectively.

We first take a look at some algebraic aspects of the pure and full braid groups.
To do this, we introduce the group presentation for $\mathbf{B}_{n}(\Sigma)$ due to Bellingeri and some well-known facts about braid groups. 
\begin{thm}\cite{Bel1}\label{thm:presentation}
The braid group $\mathbf{B}_{n}(\Sigma_{g,p})$ admits the following group presentation.
\begin{itemize}
\item The generators are $\sigma_1,\dots, \sigma_{n-1}, a_1,\dots, a_g, b_1,\dots, b_g, \zeta_1,\dots,\zeta_{p}$.
\item The defining relators are 
\begin{itemize}
\item [$\rm (BR_1)$] $\sigma_i\sigma_j=\sigma_j\sigma_i$ for $|i-j|>1$;
\item [$\rm (BR_2)$] $\sigma_i\sigma_j\sigma_i=\sigma_j\sigma_i\sigma_j$ for $|i-j|=1$;
\item [$\rm (CR_1)$] $[a_r, \sigma_i], [b_s, \sigma_i], [\zeta_t, \sigma_i]$ for $i>1$;
\item [$\rm (CR_2)$] $[a_r, \sigma_1 a_r\sigma_1], [b_r, \sigma_1 b_r\sigma_1], [\zeta_t, \sigma_1 \zeta_t\sigma_1]$ for $1\le r\le g, 1\le t\le p$;
\item [$\rm (CR_3)$] $[a_r, \sigma_1^{-1} a_s\sigma_1], [a_r, \sigma_1^{-1} b_s \sigma_1], [b_r, \sigma_1^{-1} a_s \sigma_1], [b_r, \sigma_1^{-1} b_s\sigma_1]$ for $r<s$, \\
$[a_r, \sigma_1^{-1}\zeta_u\sigma_1], [b_r,\sigma_1^{-1}\zeta_u\sigma_1]$ for $1\le r\le g, 1\le u\le p$,\\
$[\zeta_t,\sigma_1^{-1}\zeta_u\sigma_1]$ for $t<u$;
\item [$\rm (SCR)$] $\sigma_1 b_r\sigma_1 a_r \sigma_1=a_r\sigma_1 b_r$;
\item [$\rm (TR)$] $\left(\sigma_1\dots\sigma_{n-1}\sigma_{n-1}\dots\sigma_1\right)\left(\prod_{r=1}^g [b_r^{-1}, a_r]\right)\left( \prod_{t=1}^p\zeta_t\right)$,
\end{itemize}
where $[x,y]=x^{-1}y^{-1}xy$.
\end{itemize}
\end{thm}

The presentation above is slightly different from Bellingeri's. Indeed, the presentation above has one more generator $\zeta_p$ and relator $\rm(TR)$ when $p\ge 1$. However, this pair can be cancelled out by the obvious Tits transformation since $\rm(TR)$ involves $\zeta_p$ only once.

\begin{cor}\label{cor:quotient}
If $p\ge 1$, then
\[
\mathbf{B}_{n}(\Sigma_{g,p})/\langle\!\langle\zeta_p\rangle\!\rangle\simeq\mathbf{B}_{n}(\Sigma_{g,p-1}).
\]
\end{cor}
\begin{proof}
This follows directly from the presentation for $\mathbf{B}_{n}(\Sigma_{g,p})$ for $p\ge 2$.

When $p=1$, then the left hand side has $\rm(TR)$ without any $\zeta$. This is exactly the same as the presentation for a braid group on a closed surface, which is the right hand side.
\end{proof}

For convenience, we denote $\chi(\Sigma)=2-2g-p$ by $\chi$ and $2g+p+n$ by $\kappa$, and so we have the equality $\chi+\kappa=n+2$. It is easy to check that if $\kappa\le3$, then $\mathbf{B}_{n}(\Sigma)$ is finite or abelian, as is $\mathbf{P}_{n}(\Sigma)$. Indeed, $\mathbf{P}_{n}(\Sigma)_{/Z}=1$ if and only if $\kappa\le3$.
Hence we exclude these cases, and assume $\kappa\ge 4$ throughout this paper.

\begin{lem}\label{lem:centerBraid}
Suppose $\kappa\ge 4$.
The following properties hold.
\begin{enumerate}
\item $Z(\mathbf{B}_{n}(\Sigma))=Z(\mathbf{P}_{n}(\Sigma))$; 
\item 
$Z\left(\mathbf{B}_n(S^2)\right)=\langle \Delta^2|\Delta^4\rangle$,
$Z\left(\mathbf{B}_{n}(\mathbb{R}^2)\right)=\langle \Delta^2\rangle$,
$Z\left(\mathbf{B}_{n}(A)\right)=\langle \Delta_{\zeta}\rangle$, and 
$Z\left(\mathbf{B}_n(T)\right)=\langle \Delta_a\rangle\oplus\langle \Delta_b\rangle$,
where
\begin{align*}
\Delta^2 =(\sigma_1\dots\sigma_{n-1})^n,&\quad
\Delta_{\zeta}=(\zeta_1\sigma_1\dots\sigma_{n-1})^n,\\
\Delta_a=(a_1\sigma_1\dots\sigma_{n-1})^n, &\quad \Delta_b=(b_1\sigma_1\dots\sigma_{n-1})^n;
\end{align*}
\item 
$Z(\mathbf{B}_n(\Sigma))=1$ if $\chi<0$; 
\item $\mathbf{P}_{n}(\Sigma)\simeq\mathbf{P}_{n}(\Sigma)_{/Z} 
\times Z(\mathbf{P}_{n}(\Sigma))$ if $\chi\ge 0$; 
\item If $\chi\ge 0$, then
$\mathbf{P}_n(\Sigma)_{/Z}$ is either
$\mathbf{P}_{\kappa-3}(\Sigma_{0,3})$ if $g=0$, or 
$\mathbf{P}_{\kappa-3}(\Sigma_{1,1})$ if $g=1$.
Hence $\mathbf{P}_{n}(\Sigma)_{/Z}$ is a free group $\mathbf{F}_2$ of rank $2$ if and only if $\kappa=4$.
\end{enumerate}
\end{lem}
\begin{proof}
For (1), (2), (3) See \cite{Bir} and \cite[Proposition~1.6, \S~4]{PR1}.

Items (4) and (5) follow directly from the group presentations. See \cite[Lemma~17]{BGG} for the $g=1$ case.
\end{proof}

We consider the map $q_i:F_{n}(\Sigma)\to F_{n-1}(\Sigma)$ forgetting the $i$-th coordinate.
Then $q_i$ gives a fiber bundle structure with fiber $\Sigma_{\hat i}= \Sigma\setminus\mathbf{z} \cup \{z_i\}$.
Let $\mathbf{U}_i=\ker({q_i}_*)$. Then from the homotopy sequence, $\mathbf{U}_i = \pi_1(\Sigma_{\hat i}, z_i)$
and therefore is a free group of rank $1-\chi(\Sigma_{\hat i})=\kappa-2.$
We choose a generating set $X_i$ for $\mathbf{U}_i$ as
\[
X_i=\{A_{i,j}, \zeta_{i,t}, a_{i,r}, b_{i,r}|j\neq i\},
\]
where for $1\le i<j\le n$,
\[
A_{i,j}=A_{j,i}=\sigma_i^{-1}\dots\sigma_{j-2}^{-1}\sigma_{j-1}^2\sigma_{j-2}\dots\sigma_i
\]
and for $1\le i\le n$, $1\le g\le r$ and $1\le t\le p$,
\[
a_{i,r}=\delta_i^{-1}a_r\delta_i,\quad
b_{i,r}=\delta_i^{-1}b_r\delta_i,\quad
\zeta_{i,t}=\delta_i^{-1}\zeta_t\delta_i,
\]
where $\delta_i=(\sigma_1\dots\sigma_{i-1})$.

Then the conjugation of $\rm(TR)$ by $\delta_i$ gives
\begin{equation}
\label{eqn:ptr}
\tag{PTR} (A_{i,i+1}\dots A_{i,n} A_{1,i}\dots A_{i-1,i})\left(\prod_r [b_{i,r}^{-1}, a_{i,r}]\right)\left(\prod_t \zeta_{i,t}\right)=e,
\end{equation}
which involves only the generators in $X_i$. Therefore $X_i$ satisfies $\rm(PTR)$ and indeed, it is the only relator for $\mathbf{U}_i$.
This follows since $\mathbf{U}_i$ has the Hopfian property.

On the other hand, it is easy to see that 
\[
A_{i,i+1}=(\delta_{i+1}\delta_i)^{-1}\sigma_1^2(\delta_{i+1}\delta_i),\quad[\sigma_1a_{i,r}\sigma_1^{-1},\delta_{i+1}]=[a_{i+1,r},\delta_i]=e.
\]

Since the relator ${\rm(SCR)}$ can be written as $\sigma_1^2=[b_r,\sigma_1^{-1}a_r^{-1}\sigma_1]$, or equivalently, $\sigma_1^2=[\sigma_1b_r\sigma_1^{-1},a_r^{-1}]$, we conjugate these by $(\delta_{i+1}\delta_i)$ to obtain
\begin{align*}
A_{i,i+1}=(\delta_{i+1}\delta_i)^{-1}\sigma_1^2(\delta_{i+1}\delta_i)&=(\delta_{i+1}\delta_i)^{-1}[b_r,\sigma_1^{-1}a_r^{-1}\sigma_1](\delta_{i+1}\delta_i)\\
&=[b_{i+1,r}, A_{i,i+1}^{-1} a_{i,r}^{-1} A_{i,i+1}]\tag{$\rm PSCR_1$}\\
&=(\delta_{i+1}\delta_i)^{-1}[\sigma_1b_r\sigma_1^{-1},a_r^{-1}](\delta_{i+1}\delta_i)\\
&=[b_{i,r},a_{i+1,r}^{-1}]\tag{$\rm PSCR_2$}.
\end{align*}

Therefore in the union $X_i\cup X_{i+1}$ the following relation is satisfied.
\begin{equation}
\label{eqn:pscr}
\tag{PSCR} A_{i,i+1}=[b_{i,r},a_{i+1,r}^{-1}]=[b_{i+1,r}, A_{i,i+1}^{-1} a_{i,r}^{-1} A_{i,i+1}]
\end{equation}

Recall the map $q_{i*}:\mathbf{P}_{n}(\Sigma)\to\mathbf{P}_{n-1}(\Sigma)$, which is surjective. 
By using the generating set $X_i$ for $\mathbf{U}_i=\ker(q_{i*})$ and induction on $n$, we can find a generating set for $\mathbf{P}_{n}(\Sigma)$ even though $q_{i*}$ does not split in general.
Indeed, the set $X=\bigcup_i X_i$ generates $\mathbf{P}_{n}(\Sigma)$ and satisfies $\rm(PSCR)$ and $\rm(PTR)$.
The complete presentation for $\mathbf{P}_{n}(\Sigma)$ can be found in \cite{Bel1} as well.

The {\em centralizer $C_G(H)$} of $H$ in $G$ is defined as
\[C_G(H)=\{x\in G| xy=yx \text{ for all }y\in H\}.\]

\begin{lem}\label{lem:normalsubgroup}
Let $\mathbf{U}_i$ be the normal subgroup of $\mathbf{P}_{n}(\Sigma)$ defined as above. Then the following properties holds: 
\begin{enumerate}
\item $\mathbf{U}_i\cap Z(\mathbf{P}_{n}(\Sigma))=1$;
\item $\mathbf{U}_i\neq\mathbf{U}_j$ in $\mathbf{P}_{n}(\Sigma)$ for $i\neq j$;
\item In $\mathbf{P}_{n}(\Sigma)_{/Z}$, all $\mathbf{U}_i$'s are either identical if $\kappa=4$ or different if $\kappa\ge 5$; and
\item $C_{\mathbf{P}_{n}(\Sigma)}(\mathbf{U}_i)=Z(\mathbf{P}_{n}(\Sigma))$ and $C_{\mathbf{P}_{n}(\Sigma)_{/Z}}(\mathbf{U}_i)=1$.
\end{enumerate}
\end{lem}
\begin{rmk}
When $\mathbf{P}_n(\Sigma)$ is centerless and $\kappa=4$, then (2) and (3) seem to contradict each other.
However, it only happens when $n=1$ and the Euler characteristic $\chi(\Sigma)=-1$, and therefore (2) and (3) are vacuous.
\end{rmk}
\begin{proof}
\noindent(1) Since any nontrivial element in $Z(\mathbf{P}_{n}(\Sigma))$ maps to a nontrivial element in $Z(\mathbf{P}_{n-1}(\Sigma))$ via $q_{i*}$ by Lemma~\ref{lem:centerBraid}, this is obvious.

\noindent(2) If $\chi\le 0$, then one of $q_{i*}(\zeta_{j,1})$ or $q_{i*}(a_{j,1})$ is nontrivial and so $\mathbf{U}_i$ misses either $\zeta_{j,1}$ or $a_{j,1}$ in $\mathbf{U}_j$. 
If $\chi>0$, then $n\ge3$ since $\kappa\ge 4$, and there exists $1\le k\le n$ different from $i$ and $j$. Then $q_{i*}(A_{j,k})$ is nontrivial and so $\mathbf{U}_i$ misses $A_{j,k}\in\mathbf{U}_j$.

\noindent(3) If $\chi<0$, then it is obvious since the center is trivial, and so we assume that $\chi\ge 0$. 
If $\kappa=4$, then 
by Lemma~\ref{lem:centerBraid}~(5), $\mathbf{P}_{n}(\Sigma)_{/Z}$ is a free group of rank 2 which is isomorphic to $\mathbf{U}_i$. 
Otherwise, one can regard $\{z_i,z_j\}$ as a part of a basepoint for $\mathbf{P}_{n}(\Sigma)_{/Z}$.
Then $\mathbf{U}_i$ is the kernel of $q_{i*}:\mathbf{P}_{n}(\Sigma)_{/Z}\to\mathbf{P}_{n-1}(\Sigma)_{/Z}$.
Hence by (2), $\mathbf{U}_i\neq\mathbf{U}_j$.

\noindent(4) We use induction on $n$. If $n=1$, then $\mathbf{P}_{1}(\Sigma)=\mathbf{U}_1$ and so $C_{\mathbf{P}_{1}(\Sigma)}(\mathbf{U}_1)=Z(\mathbf{P}_{1}(\Sigma))$.

Suppose the statement holds for $n$ and let $\alpha\in C_{\mathbf{P}_{n+1}(\Sigma)}(\mathbf{U}_i)$.
Then by taking $q_{j*}$ for $j\neq i$, we have
$q_{j*}(\alpha)\in C_{\mathbf{P}_{n}(\Sigma)}(q_{j*}(\mathbf{U}_i))=C_{\mathbf{P}_{n}(\Sigma)}(\mathbf{U}_i)=Z(\mathbf{P}_{n}(\mathbf{U}_i))$ by the induction hypothesis.
Since $q_{j*}$ maps $Z(\mathbf{P}_{n+1}(\Sigma))$ onto $Z(\mathbf{P}_{n}(\Sigma))$, there exists $z\in Z(\mathbf{P}_{n+1}(\mathbf{U}_j))$ and $\alpha'\in\mathbf{U}_j$ so that $\alpha = z \alpha'$.

However, if we take $\beta=A_{i,j}\in\mathbf{U}_i\cap\mathbf{U}_j$, then the commutativity of $\alpha'$ and $\beta$ implies that $\alpha'$ must be a power of $\beta$ by the freeness of $\mathbf{U}_j$, and so contained in $\mathbf{U}_i$.
This implies that $\alpha'\in Z(\mathbf{U}_i)=1$ as well. Hence $\alpha\in Z(\mathbf{P}_{n+1}(\Sigma))$.

The second assertion is obvious since $\mathbf{P}_{n}(\Sigma)_{/Z}$ is isomorphic to a pure braid group on a surface with $\chi<0$ which is centerless by Lemma~\ref{lem:centerBraid}~(3).
\end{proof}

We introduce another remarkable result due to Goldberg as follows.
\begin{thm}\label{thm:Goldberg}\cite{Go} Suppose $\Sigma\neq S^2$.
Then the map $\iota_*:\mathbf{P}_{n}(\Sigma)\to\prod^n_{i=1} \pi_1(\Sigma, z_i)$ induced by the embedding $\iota:F_{n}(\Sigma)\to\prod^n\Sigma$ is surjective and
\[
\ker \iota_* = \langle\!\langle  A_{i,j}|1\le i<j\le n \rangle\!\rangle,
\]
which is the normal closure generated by $A_{i,j}$'s.
\end{thm}
Roughly speaking, all the interference between strands in $\mathbf{P}_{n}(\Sigma)$ come from the $A_{i,j}$'s.

Now we review the relationship between braid groups and mapping class groups.
For convenience, we denote by $fg$ for the composition $f\circ g$ of two mapping classes or automorphisms $f$ and $g$, and apply group elements from left to right. That is, $(fg)(x)=g(f(x))$.

Let us consider the map $\mu\times\rho:\mathcal{M}^*_{n}(\Sigma)\to\mathbf{S}_p\times\mathbf{S}_n$ defined by
\[
\mu(f)(t)=u\Longleftrightarrow f(p_t)=p_u,\quad
\rho(f)(i)=j\Longleftrightarrow f(z_i)=z_j.
\]
We denote the kernel $\ker(\rho)$ by $\mathcal{PM}^*_{n}(\Sigma)$. 

Recall $Push:\mathbf{B}_{n}(\Sigma)_{/Z}\to\mathcal{M}^*_{n}(\Sigma)$ in the exact sequence of Birman.
Then it is obvious that $\rho(\beta)=\rho(Push(\beta))$, and so 
$Push|:\mathbf{P}_{n}(\Sigma)_{/Z}\to\mathcal{PM}^*_{n}(\Sigma)$ is well-defined.
Hence we may regard $\mathbf{B}_{n}(\Sigma)_{/Z}$ and $\mathbf{P}_{n}(\Sigma)_{/Z}$ as normal subgroups of $\mathcal{M}^*_{n}(\Sigma)$ and $\mathcal{PM}^*_{n}(\Sigma)$ via $Push$ and $Push|$, respectively. 
According to the context, we may regard the domains of $Push$ and $Push|$ as $\mathbf{B}_{n}(\Sigma)$ and $\mathbf{P}_{n}(\Sigma)$, respectively, by the pre-compositions of the canonical quotient map by their centers.

On the other hand, any mapping class $f\in\mathcal{M}^*_{n}(\Sigma)$ induces self-homeomorphisms on both ordered and unordered configuration spaces $F_n(\Sigma)$ and $F_n(\Sigma)/\mathbf{S}_n$. Therefore they induce automorphisms $f_*$ on $\mathbf{B}_{n}(\Sigma)$ and $\mathbf{P}_{n}(\Sigma)$ simultaneously, and so 
there are two maps 
\[
(\cdot)_*:\mathcal{M}^*_{n}(\Sigma)\to\Aut(\mathbf{B}_{n}(\Sigma)),\quad
(\cdot)_*|:\mathcal{M}^*_{n}(\Sigma)\to\Aut(\mathbf{P}_{n}(\Sigma)).
\]
\begin{rmk}
For $f\in\mathcal{M}^*_n(\Sigma)$ and $\beta\in\mathbf{B}_n(\Sigma)$, the mapping class $Push(f_*(\beta))$ is nothing but $f^{-1}Push(\beta) f$.
However, $f_*(\beta)$ is not a conjugate of $\beta$ in general.
Indeed, for a torus or higher genus surface, Dehn twists along non-separating curves never act by conjugacy on the braid group.
\end{rmk}

\begin{lem}\label{lem:kernel}
The maps $(\cdot)_*$ and $(\cdot)_*|$ are injective.
\end{lem}
\begin{proof}
Suppose $f_*=Id$ on $\mathbf{B}_{n}(\Sigma)$ or $\mathbf{P}_{n}(\Sigma)$. 
Then $\rho(f)=Id_{\mathbf{z}}$ since $f_*(\mathbf{U}_i)=\mathbf{U}_j$ for $j=\rho(f)(i)$ and all $\mathbf{U}_i$'s are pairwise different by Lemma~\ref{lem:normalsubgroup}~(2).
Therefore $f$ can be regarded as an element in $\mathcal{M}^*(\Sigma_{\hat i},z_i)$, and so
it is enough to show that $f$ is trivial in $\mathcal{M}^*(\Sigma_{\hat i}, z_i)$ because both $\mathcal{M}^*(\Sigma_{\hat i}, z_i)$ and $\mathcal{M}^*_{n}(\Sigma)$ are subgroups of $\mathcal{M}^*_{p+n}(\Sigma_g)$.

Since $\Sigma_{\hat i}$ has the Euler characteristic  $\chi(\Sigma_{\hat i}) = 2-\kappa<0$ and so it is $K(\pi,1)$, there is a one-to-one correspondence between homotopy classes of maps on $(\Sigma_{\hat i},z_i)$ and endomorphisms on $\pi_1(\Sigma_{\hat i},z_i)=\mathbf{U}_i$.
This implies that $f$ is homotopic to the identity.
However any two homotopic homeomorphisms are isotopic on any surface with negative Euler characteristic \cite[Theorem~1.9]{FM}, and therefore $f$ is trivial in $\mathcal{M}^*(\Sigma_{\hat i},z_i)$.
\end{proof}

The above lemma describes the geometry of the automorphism groups, and we will focus on the algebraic part of them later.

\section{Automorphism groups of braid groups}
\subsection{Transvection subgroups}
In this section, we concerns the cases when $\chi(\Sigma)\ge 0$. Otherwise the center $Z(\mathbf{B}_n(\Sigma))$ is trivial and so is the transvection subgroup $tv(\mathbf{B}_n(\Sigma))$.
We consider two exact sequences as follows.
\begin{align*}
1&\longrightarrow\tv(\mathbf{B}_{n}(\Sigma))\longrightarrow
\Aut(\mathbf{B}_{n}(\Sigma))\stackrel{\Psi}{\longrightarrow}\Aut(\mathbf{B}_{n}(\Sigma)_{/Z}),\\
1&\longrightarrow\tv(\mathbf{P}_{n}(\Sigma))\longrightarrow\Aut(\mathbf{P}_{n}(\Sigma))\stackrel{\Phi}{\longrightarrow}\Aut(\mathbf{P}_{n}(\Sigma)_{/Z})\longrightarrow1.
\end{align*}

The maps $\Phi$ and $\Psi$ are obvious and the exactness at the center of each row is just the definition of the transvection subgroup
$\tv(G)=\ker(\Aut(G)\to\Aut(G_{/Z})).$
Note that $\phi\in\tv(G)$ if and only if for any $g\in G$, there exists $z_g\in Z(G)$ such that $\phi(g)=g z_g$.

However, the surjectivity of $\Phi$ is not quite obvious. 
For any $\phi\in\Aut(\mathbf{P}_{n}(\Sigma)_{/Z})$, we consider the automorphism $\phi\times Id$ on $\mathbf{P}_n(\Sigma)_{/Z}\times Z(\mathbf{P}_n(\Sigma))$. 
Since $\mathbf{P}_{n}(\Sigma)\simeq\mathbf{P}_{n}(\Sigma)_{/Z}\times Z(\mathbf{P}_{n}(\Sigma))$ by Lemma~\ref{lem:centerBraid}~(4),
we may regard $\phi\times Id$ as an element in $\Aut(\mathbf{P}_n(\Sigma))$, which is in the preimage of $\phi$ under $\Phi$. Hence $\Phi$ is surjective and splits, and so
\[
\Aut(\mathbf{P}_{n}(\Sigma))\simeq \tv(\mathbf{P}_{n}(\Sigma))\rtimes\Aut(\mathbf{P}_{n}(\Sigma)_{/Z}).
\]

Since the center is a characteristic subgroup, the restriction map $(\cdot)|_Z:\tv(\mathbf{P}_{n}(\Sigma))\to\Aut(Z(\mathbf{P}_{n}(\Sigma)))$ is well-defined. Moreover, it is surjective and splits since $(Id\times \phi_Z)|_Z=\phi_Z$ for any $\phi_Z\in\Aut(Z(\mathbf{P}_{n}(\Sigma)))$.

Let $\phi\in\ker(\cdot)|_Z$. Then for any $\alpha\in\mathbf{P}_{n}(\Sigma)_{/Z}$ and $z\in Z(\mathbf{P}_{n}(\Sigma))$, there exists $z'$ such that
$\phi(\alpha z)=\phi(\alpha)\phi(z)=(\alpha z')z.$
Therefore $\phi$ can be regarded as an element in $\operatorname{Hom}(\mathbf{P}_{n}(\Sigma)_{/Z}, Z(\mathbf{P}_{n}(\Sigma)))$, and so
\begin{equation}\label{eqn:generaltransvection}\tag{TV}
\tv(\mathbf{P}_{n}(\Sigma))\simeq H^1\left(\mathbf{P}_{n}(\Sigma)_{/Z}; Z(\mathbf{P}_{n}(\Sigma))\right)\rtimes\Aut(Z(\mathbf{P}_{n}(\Sigma))).
\end{equation}

Zhang asserted in \cite[Lemma~2.1]{Zh2} that for any group $G$, the map $\Aut(G)\to\Aut(G_{/Z})$ is always surjective. However this is not true in general, in particular when $G$ is not a split central extension of $G_{/Z}$.
Indeed, $\mathbf{B}_4(S^2)_{/Z}$ is a subgroup of $\mathcal{M}_4^*(S^2)$ of index $2=|\mathcal{M}^*(S^2)|$, and 
\[
\Aut(\mathbf{B}_4(S^2))\simeq\tv(\mathbf{B}_4(S^2))\rtimes \mathcal{M}_4^*(S^2)
\]
due to Bellingeri \cite[Theorem~6.2]{Bel2}. Moreover, one can show that  $\Aut(\mathbf{B}_4(S^2)_{/Z})\simeq\mathcal{M}^*_4(S^2)\rtimes\mathbb{Z}_2$ by the direct computation.
However,
\[
\Psi(\Aut(\mathbf{B}_4(S^2)))=\mathcal{M}_4^*(S^2)\subsetneq
\mathcal{M}^*_4(S^2)\rtimes\mathbb{Z}_2\simeq \Aut(\mathbf{B}_4(S^2)_{/Z}).
\]
Therefore $\mathbf{B}_4(S^2)$ is a counter example to Zhang's assertion.

For $\Sigma=S^2$, $\mathbb{R}^2$ and $A$, both $\tv(\mathbf{P}_{n}(\Sigma))$ and $\tv(\mathbf{B}_{n}(\Sigma))$ are known by Charney and Crisp \cite{CC}, Bell and Margalit \cite{BM}, or can be computed as follows.

Let $N=\binom{n+p-1}2 -1$. Then
\begin{enumerate}
\item 
$\tv(\mathbf{P}_n(S^2))\simeq \mathbb{Z}_2^{N}$ and $\tv(\mathbf{B}_n(S^2))\simeq\mathbb{Z}_2$ for $n\ge 4$;
\item 
$\tv(\mathbf{P}_{n}(\mathbb{R}^2))\simeq\mathbb{Z}^N\rtimes\mathbb{Z}_2$ and $\tv(\mathbf{B}_{n}(\mathbb{R}^2))=1$ for $n\ge 3$;
\item 
$\tv(\mathbf{P}_{n}(A))\simeq\mathbb{Z}^N\rtimes\mathbb{Z}_2$ for $n\ge 2$ and $\tv(\mathbf{B}_{n}(A))\simeq \begin{cases}\mathbb{Z}&n\ge3;\\ \mathbb{Z}\rtimes\mathbb{Z}_2& n=2.\end{cases}$
\end{enumerate}

On the other hand, it is quite complicated to calculate the transvection subgroups for $\Sigma=T$ because $Z(\mathbf{B}_n(T))$ has rank 2.
The computation of $\tv(\mathbf{P}_n(T))$ and $\tv(\mathbf{B}_n(T))$ is given in the following lemma.
\begin{lem}\label{lem:transvection}
The transvection subgroups $\tv(\mathbf{P}_n(T))$ and $\tv(\mathbf{B}_n(T))$ are as follows.
\begin{enumerate}
\item $\tv(\mathbf{P}_n(T))\simeq\mathbb{Z}^{4(n-1)}\rtimes GL(2,\mathbb{Z})$ for $n\ge 2$;
\item $\tv(\mathbf{B}_n(T))\simeq GL(2,\mathbb{Z})[n]$ for $n\ge 2$, where $GL(2,\mathbb{Z})[n]$ is the $n$-congruence subgroup of $GL(2,\mathbb{Z})$.
\end{enumerate}
\end{lem}
\begin{proof}
\noindent (1) 
Since $H_1(\mathbf{P}_n(T))=H_1(\mathbf{P}_n(T)_{/Z})\times Z(\mathbf{P}_n(T))$, we first compute $H_1(\mathbf{P}_n(T))$.
By (\ref{eqn:pscr}), $[A_{i,j}]=0$ in $H_1(\mathbf{P}_n(T))$ and therefore the abelianization factors as follows:
\[
\mathbf{P}_n(T)\to\mathbf{P}_n(T)/\langle\!\langle A_{i,j}\rangle\!\rangle\to H_1(\mathbf{P}_n(T)).
\]
However, the group in the middle is isomorphic to $\prod^n \pi_1(T)\simeq \mathbb{Z}^{2n}$ by Theorem~\ref{thm:Goldberg}, which is already abelian and generated by $[a_{i,1}]$'s and $[b_{i,1}]$'s, and therefore the map on the right must be an isomorphism.
Moreover, since $Z(\mathbf{P}_n(T))\simeq \mathbb{Z}^2$ generated by 
\[
\Delta_a=a_{1,1}\dots a_{n,1},\quad \Delta_b=b_{1,1}\dots b_{n,1},
\]
one each of the $[a_{i,1}]$'s and $[b_{i,1}]$'s, say $[a_{n,1}]$ and $[b_{n,1}]$ are redundant in $H_1(\mathbf{P}_n(T)_{/Z})$.
Hence we have $H_1(\mathbf{P}_n(T)_{/Z})\simeq \mathbb{Z}^{2(n-1)}$, and so
\[
H^1(\mathbf{P}_n(T)_{/Z};Z(\mathbf{P}_n(T)))\simeq \mathbb{Z}^{4(n-1)}.
\]

Finally, since $\Aut(Z(\mathbf{P}_n(T)))\simeq GL(2,\mathbb{Z})$, the equation (\ref{eqn:generaltransvection}) completes the proof.

\noindent (2) Let $\phi\in\tv(\mathbf{B}_n(T))$.
Then the braid relations $\rm(BR_1), (BR_2)$ imply that there exists $z\in Z(\mathbf{B}_n(T))$ so that $\phi(\sigma_i)$ is $\sigma_i z$ for all $i$, and moreover $\rm(SCR)$ forces $z^2$ to be trivial. 
Then $z$ must be trivial since $Z(\mathbf{B}_n(T))$ is torsion-free.
Hence there exist integers $x_1, y_1, x_2$ and $y_2$ such that
\[
\phi(a_1)=a_1\Delta_a^{x_1}\Delta_b^{y_1},\quad 
\phi(b_1)=b_1\Delta_a^{x_2}\Delta_b^{y_2}.
\]

From this we have
\[\phi(\Delta_a)=\phi(a_1\sigma_1\dots\sigma_{n-1})^n=\Delta_a^{nx_1+1}\Delta_b^{ny_1},\]
\[\phi(\Delta_b)=\phi(b_1\sigma_1\dots\sigma_{n-1})^n=\Delta_a^{nx_2}\Delta_b^{ny_2+1}.\]
Hence $\phi|_{Z(\mathbf{B}_n(T))}$ can be represented by a $2\times2$ matrix
$M_\phi=\left(
\begin{matrix}
nx_1+1& nx_2\\
ny_1&ny_2+1
\end{matrix}\right)$ with determinant $\pm1$ since $M_{\phi}\in GL(2,\mathbb{Z})\simeq\Aut(Z)$, and so $M_\phi\in GL(2,\mathbb{Z})[n]$.

Conversely, for any matrix $M\in GL(2,\mathbb{Z})[n]$, one can find integers $x_1, y_1, x_2$ and $y_2$ which define an automorphism in $\tv(\mathbf{B}_n(T))$ by using these parameters.
\end{proof}

Note that the automorphism $\phi\in\tv(\mathbf{B}_2(T))$ with $M_\phi=-I$ is induced by the hyperelliptic involution $s\in Z(\mathcal{M}^*_2(T))$ interchanging points of $\mathbf{z}$.
Hence $s\in\tv(\mathbf{B}_2(T))\cap (\cdot)_*(\mathcal{M}^*_2(T))$.

In general, it is easy to check that
\begin{align*}
Z(\mathcal{M}^*_n(\Sigma))&\subset C_{\mathcal{M}^*_n(\Sigma)}(\mathbf{B}_n(\Sigma)_{/Z}) \simeq \tv(\mathbf{B}_n(\Sigma))\cap (\cdot)_*(\mathcal{M}^*_n(\Sigma))\\
&\subset C_{\mathcal{M}^*_n(\Sigma)}(\mathbf{P}_n(\Sigma)_{/Z}) \simeq \tv(\mathbf{P}_n(\Sigma))\cap (\cdot)_*(\mathcal{M}^*_n(\Sigma)).
\end{align*}

We claim that all inclusions above are indeed bijective as follows.

\begin{lem}\label{lem:intersection}Suppose $\kappa\ge 4$. Then the centralizer $C_{\mathcal{M}^*_n(\Sigma)}(\mathbf{P}_n(\Sigma)_{/Z})$ is isomorphic to the center $Z(\mathcal{M}^*_n(\Sigma))$,
which is either $\mathbb{Z}_2$ if $(\Sigma,n)$ is $(A,2)$ or $(T,2)$, or trivial otherwise.
\end{lem}
\begin{proof}
There is nothing to prove for $\chi(\Sigma)<0$ since $\tv(\mathbf{P}_n(\Sigma))=1$.

Suppose $\chi(\Sigma)\ge 0$.
If $n\ge3$, then for any $f\in C_{\mathcal{M}^*_n(\Sigma)}(\mathbf{P}_n(\Sigma)_{/Z})$, the induced permutation $\rho(f)$ is the identity on $\mathbf{z}$ since $\rho(f)$ is in the center of the symmetric group $\mathbf{S}_n$ which is trivial. Hence
\[
C_{\mathcal{M}^*_n(\Sigma)}(\mathbf{P}_n(\Sigma)_{/Z})=C_{\mathcal{PM}^*_n(\Sigma)}(\mathbf{P}_n(\Sigma)_{/Z}).
\]
However, $\mathbf{P}_n(\Sigma)_{/Z}\simeq\mathbf{P}_{\kappa-3}(\Sigma')$ where $\Sigma'$ is either $\Sigma_{0,3}$ or $\Sigma_{1,1}$ by Lemma~\ref{lem:centerBraid}~(5), and $\mathcal{PM}^*_n(\Sigma)\subset \mathcal{M}^*_{\kappa-3}(\Sigma')$.
Therefore,
\[
C_{\mathcal{PM}^*_n(\Sigma)}(\mathbf{P}_n(\Sigma)_{/Z})\subset C_{\mathcal{M}^*_{\kappa-3}(\Sigma')}(\mathbf{P}_{\kappa-3}(\Sigma'))
\]
which is trivial since $\chi(\Sigma')=-1$.

Therefore, the only case we have to consider is when $\chi(\Sigma)\ge 0$ and $n\le 2$. By the constraint of $\kappa\ge4$, there are only two possibilities, namely the 2-braid groups on $A$ and $T$. In each case, the center $Z(\mathcal{M}^*_2(\Sigma))$ is isomorphic to $\mathbb{Z}_2$ and
\[
\mathcal{M}^*_2(\Sigma)= Z(\mathcal{M}^*_2(\Sigma))\times\mathcal{M}^*_2(\Sigma)_{/Z},\quad
\mathcal{M}^*_2(\Sigma)_{/Z}\simeq\mathcal{PM}^*_2(\Sigma).
\]
Hence 
\[
C_{\mathcal{M}^*_2(\Sigma)}(\mathbf{P}_2(\Sigma)_{/Z})=
Z(\mathcal{M}^*_2(\Sigma))\times C_{\mathcal{PM}^*_2(\Sigma)}(\mathbf{P}_2(\Sigma)_{/Z}).
\]

Since the rightmost factor is trivial by the same argument as above, this completes the proof.
\end{proof}

Since $\tv(\mathbf{B}_n(\Sigma))\cap \Inn(\mathbf{B}_n(\Sigma))=1$, the subgroup $\tv(\mathbf{B}_n(\Sigma))\cap(\cdot)^*\mathcal{M}^*(\Sigma)$ of $\Out(\mathbf{B}_n(\Sigma))$ is also isomorphic to $Z(\mathcal{M}^*_n(\Sigma))$. So we may identify $Z(\mathcal{M}^*_n(\Sigma))$ with this subgroup of $\mathcal{M}^*(\Sigma)$.

\begin{cor}
Suppose $\kappa\ge4$. Then $\tv(\mathbf{B}_n(\Sigma))\rtimes\mathcal{M}^*_n(\Sigma)/Z$ and $\tv(\mathbf{B}_n(\Sigma))\rtimes\mathcal{M}^*(\Sigma)/Z$ are isomorphic to subgroups of $\Aut(\mathbf{B}_n(\Sigma))$ and $\Out(\mathbf{B}_n(\Sigma))$, respectively.
\end{cor}
\begin{proof}
This follows directly from Lemma~\ref{lem:kernel} and Lemma~\ref{lem:intersection}.
\end{proof}

\subsection{Automorphism groups}
We first consider $\Aut(\mathbf{P}_{n}(\Sigma)_{/Z})$.
Recall that $\mathbf{P}_n(\Sigma)_{/Z}=1$ if and only if $\kappa\le3$, and there is nothing to do.
If $\kappa=4$, then $\Aut(\mathbf{P}_{n}(\Sigma)_{/Z})\simeq\Aut(\mathbf{F}_2)$ by Lemma~\ref{lem:centerBraid}~(5).

For $\kappa\ge 5$, all known results can be summarized as follows.
\begin{thm}\cite{Bel2, CC, IIM, KY}\label{thm:kappa5}
Suppose $\kappa\ge 5$. If $\chi\ge -1$ or $p=0$, then
\[\Aut(\mathbf{P}_{n}(\Sigma_{g,p})_{/Z})\simeq\mathcal{M}^*_{p+n}(\Sigma_g).\]
\end{thm}
\begin{rmk}\label{rmk:exceptionalPn}
The cases for $\chi\ge -1$ are covered by \cite{Bel2} for $g=p=0$ and by \cite{CC} for $g=0$ and $p\ge 1$ with the aid of Theorem~\ref{thm:AutMCG} due to Ivanov and Korkmaz \cite{Iv2, K}. 

If $p=0$, then the cases for $T$ with $n\ge 3$ and $\Sigma_g$ with $g\ge 2,n\ge 3$ are covered by \cite{IIM}, and the cases with $g\ge 2, n=2$ are covered by \cite{KY}.

Furthermore, the above discussion for $\kappa\le 4$ and theorem cover all the exceptional cases. Hence there is no unknown exceptional case for $\Aut(\mathbf{P}_{n}(\Sigma))$.
\end{rmk}

When $\chi<-1$, we have the following proposition which is the first half of Theorem~\ref{thm:AutPn} and will be proved later.
\begin{prop}\label{prop:AutPn}
Suppose $\chi<-1$. Then for any $n\ge 2$ and $p\ge 0$,
\[\Aut(\mathbf{P}_{n}(\Sigma))\simeq\mathcal{M}^*_{n}(\Sigma).\]
\end{prop}
Note that as mentioned before, all closed cases with $g\ge 2$ have been already treated in \cite{Bel2, IIM} and in \cite{KY}.

Using this proposition, we have the following theorem, which is nothing but a reorganization of known results and the above proposition.

\begin{thm}
For any $g, p\ge 0$ and $n\ge 2$ with $\kappa\ge 4$,
\[
\Aut(\mathbf{P}_{n}(\Sigma))\simeq
\begin{cases}
\tv(\mathbf{P}_{n}(\Sigma))\rtimes \Aut(\mathbf{F}_2)& \kappa=4;\\
\tv(\mathbf{P}_{n}(\Sigma))\rtimes\mathcal{M}^*_{p+n}(\Sigma_g) & \chi\ge -1, \kappa\ge 5;\\
\mathcal{M}^*_{n}(\Sigma)& \chi<-1,
\end{cases}
\]
and
\[
\Out(\mathbf{P}_{n}(\Sigma))\simeq
\begin{cases}
\tv(\mathbf{P}_{n}(\Sigma))\rtimes GL(2,\mathbb{Z})
& \kappa=4;\\
\tv(\mathbf{P}_{n}(\Sigma))\rtimes(\mathbf{S}_{p+n}\times\mathcal{M}^*(\Sigma_g)) & \chi\ge -1, \kappa\ge 5;\\
\mathbf{S}_n\times\mathcal{M}^*(\Sigma)& \chi<-1.
\end{cases}
\]
\end{thm}

\begin{proof}
The only nontrivial statement is about $\Out(\mathbf{P}_n(\Sigma))$.
For $\kappa=4$, this statement follows from the facts that $\Inn(\mathbf{P}_{n}(\Sigma))\simeq\mathbf{F}_2\subset\Aut(\mathbf{F}_2)$ and $\Out(\mathbf{F}_2)\simeq GL(2,\mathbb{Z})$.

If $\kappa\ge 5$ and $\chi\ge -1$, then 
\[\Inn(\mathbf{P}_{n}(\Sigma))\simeq\mathbf{P}_{n}(\Sigma)_{/Z}\simeq\mathbf{P}_{n+p}(\Sigma_g)_{/Z}\subset\mathcal{M}^*_{n+p}(\Sigma_g)\]
by Lemma~\ref{lem:centerBraid}~(5) and the Birman exact sequence. 

Recall $\mathcal{PM}^*_{n+p}(\Sigma_g)$ which is the kernel of the induced permutation $\rho:\mathcal{M}^*_{n+p}(\Sigma_g)\to\mathbf{S}_{n+p}$.
Then we have two triples,
\[\mathbf{P}_{n+p}(\Sigma_g)_{/Z}\subset\mathbf{B}_{n+p}(\Sigma_g)_{/Z}\xrightarrow{Push}\mathcal{M}^*_{n+p}(\Sigma_g)\] and
\[\mathbf{P}_{n+p}(\Sigma_g)_{/Z}\xrightarrow{Push|}\mathcal{PM}^*_{n+p}(\Sigma_g)\subset\mathcal{M}^*_{n+p}(\Sigma_g)\]
which make the commutative diagram below.
\[
\xymatrix{
\mathbf{S}_{n+p} & \frac{\mathcal{M}^*_{n+p}(\Sigma_g)}{\mathcal{PM}^*_{n+p}(\Sigma_g)}\ar[l]_{\simeq}\\
\frac{\mathbf{B}_{n+p}(\Sigma_g)_{/Z}}{\mathbf{P}_{n+p}(\Sigma_g)_{/Z}}\ar@{^(->}[r]\ar[u]^{\simeq}
&\frac{\mathcal{M}^*_{n+p}(\Sigma_g)}{\mathbf{P}_{n+p}(\Sigma_g)_{/Z}}\ar@{->>}[r]\ar@{->>}[u]
&\frac{\mathcal{M}^*_{n+p}(\Sigma_g)}{\mathbf{B}_{n+p}(\Sigma_g)_{/Z}}\ar[d]^{\simeq}\\
&\frac{\mathcal{PM}^*_{n+p}(\Sigma_g)}{\mathbf{P}_{n+p}(\Sigma_g)_{/Z}}\ar@{^(->}[u]\ar[r]^{\simeq}&\mathcal{M}^*(\Sigma_g)
}
\]

The isomorphisms on the top and left come from the induced permutations, and those on the right and below come from the Birman exact sequence.
Hence all morphisms split and therefore
\[
\frac{\mathcal{M}^*_{n+p}(\Sigma_g)}{\mathbf{P}_{n+p}(\Sigma_g)_{/Z}}\simeq\mathbf{S}_{n+p}\times\mathcal{M}^*(\Sigma_g).
\]

If $\chi<-1$, $\Out(\mathbf{P}_n(\Sigma))\simeq \mathcal{M}^*_n(\Sigma)/\mathbf{P}_n(\Sigma)$ since $\Aut(\mathbf{P}_n(\Sigma))\simeq\mathcal{M}^*_n(\Sigma)$.
Then two triples 
\[
\mathbf{P}_n(\Sigma)\subset\mathbf{B}_n(\Sigma)\xrightarrow{Push}\mathcal{M}^*_n(\Sigma),\quad
\mathbf{P}_n(\Sigma)\xrightarrow{Push|}\mathcal{PM}_n^*(\Sigma)\subset\mathcal{M}^*_n(\Sigma)
\]
produce a similar commutative diagram as above, and therefore 
\[
\Out(\mathbf{P}_n(\Sigma))\simeq\mathbf{S}_n\times\mathcal{M}^*_n(\Sigma)
\]
as desired.
\end{proof}

Suppose that $\mathbf{P}_{n}(\Sigma)$ is a characteristic subgroup of $\mathbf{B}_{n}(\Sigma)$.
Then there is a canonical map $\Gamma:\Aut(\mathbf{B}_{n}(\Sigma))\to\Aut(\mathbf{P}_{n}(\Sigma))$ defined by restriction, and $(\cdot)_*|$ factors through $\Gamma$. That is,
\[
(\cdot)_*|=\Gamma\circ(\cdot)_*:\mathcal{M}^*_{n}(\Sigma)\to\Aut(\mathbf{B}_{n}(\Sigma))\to\Aut(\mathbf{P}_{n}(\Sigma)).
\]
Moreover, $\Gamma$ induces a map $\bar\Gamma:\Aut(\mathbf{B}_{n}(\Sigma)_{/Z})\to\Aut(\mathbf{P}_{n}(\Sigma)_{/Z})$ such that 
$\bar\Gamma\circ\Psi = \Phi\circ\Gamma$.

\begin{lem}\label{lem:vertInj}
Suppose $\mathbf{P}_{n}(\Sigma)$ is a characteristic subgroup of $\mathbf{B}_{n}(\Sigma)$ and $\Sigma\neq S^2$.
Then both $\Gamma$ and $\bar\Gamma$ are injective.
\end{lem}
\begin{proof}
Let $\phi\in\ker(\Gamma)$ and $\beta\in\mathbf{B}_{n}(\Sigma)$.
Then for any $\gamma\in\mathbf{P}_{n}(\Sigma)$, $\phi(\beta^{-1}\gamma\beta) =\beta^{-1}\gamma\beta$ because $\mathbf{P}_{n}(\Sigma)$ is a normal subgroup of $\mathbf{B}_{n}(\Sigma)$. Therefore 
\[
\phi(\beta)\beta^{-1}\in C_{\mathbf{B}_{n}(\Sigma)}(\mathbf{P}_{n}(\Sigma))=Z(\mathbf{P}_{n}(\Sigma)),
\]
and so there exists $z_\beta\in Z(\mathbf{P}_{n}(\Sigma))$ such that $\phi(\beta)=\beta z_\beta$.

This implies that $\phi(\beta)^m = \beta^m z_\beta^m = \beta^m$ for some $m\ge 1$ since $\mathbf{P}_{n}(\Sigma)$ is a finite index subgroup of $\mathbf{B}_{n}(\Sigma)$, and so $z_\beta^m=e$.
Since the braid group $\mathbf{B}_{n}(\Sigma)$ is torsion free unless $\chi=2$, the element $z_\beta$ must be trivial.
Therefore $\Gamma$ is injective. 

The injectivity of $\bar\Gamma$ follows from the same argument as above.
\end{proof}

\begin{cor}
Suppose $\mathbf{P}_{n}(\Sigma)$ is a characteristic subgroup of $\mathbf{B}_{n}(\Sigma)$ and $\chi<-1$. Then
\[
\Aut(\mathbf{B}_{n}(\Sigma))\simeq\mathcal{M}^*_{n}(\Sigma),\quad
\Out(\mathbf{B}_{n}(\Sigma))\simeq\mathcal{M}^*(\Sigma).
\]
\end{cor}
\begin{proof}
Notice that for a generic case, $(\cdot)_*|=\Gamma\circ(\cdot)_*$ is an isomorphism by Proposition~\ref{prop:AutPn}, and therefore $\Gamma$ should be surjective.

On the other hand, by Lemma~\ref{lem:vertInj}, $\Gamma$ is also injective. Hence for a generic case, $\Gamma$ becomes an isomorphism, and so $\Aut(\mathbf{B}_{n}(\Sigma))\simeq\mathcal{M}^*_{n}(\Sigma)$.
The outer automorphism group follows directly from the Birman exact sequence.
\end{proof}

Therefore, this corollary together with Theorem~\ref{thm:Characteristic} is the second half of Theorem~\ref{thm:AutPn}. 
After proving Proposition~\ref{prop:AutPn} and Theorem~\ref{thm:Characteristic}, we will consider all exceptional cases as well in the final section.

\section{Proof of Proposition~\ref{prop:AutPn}}

In this section, we assume $\chi< -1$ and $n\ge 2$, and identify $\mathbf{P}_{n}(\Sigma)$ with a subgroup of $\mathcal{M}^*_{n}(\Sigma)$ via $Push$.
Furthermore, let $\phi$ be an automorphism on $\mathbf{P}_{n}(\Sigma)$, and we denote by $\bar\phi$ the induced automorphism on the abelianization $H_1(\mathbf{P}_{n}(\Sigma))$.

\subsection{Outline of the proof}
Each $\mathbf{U}_i$ plays an essential role in the proof of Theorem~\ref{thm:AutPn} because the action of $\phi$ on $\mathbf{P}_{n}(\Sigma)$ is determined by how $\phi$ acts on $\mathbf{U}_i$. 
More precisely, we have the following lemma which is essentially the same as Proposition~11.3 in \cite{Iv1} and Theorem~3 in \cite{Zh1}.

\begin{lem}\label{lem:AutUi}
Suppose that there exists $1\le i\le n$ such that $\phi(\mathbf{U}_i)=\mathbf{U}_i$.
Then $\phi|_{\mathbf{U}_i}=Id_{\mathbf{U}_i}$ if and only if $\phi$ is the identity.
\end{lem}
\begin{proof} The proof is similar to Lemma~\ref{lem:vertInj}.

Let $\beta\in\mathbf{P}_{n}(\Sigma)$ and $\gamma\in\mathbf{U}_i$.
Since $\mathbf{U}_i$ is normal in $\mathbf{P}_{n}(\Sigma)$, $\phi(\beta^{-1}\gamma\beta)=\beta^{-1}\gamma\beta,$
and so $\phi(\beta)\beta^{-1}\in C_{\mathbf{P}_{n}(\Sigma)}(\mathbf{U}_i)=Z(\mathbf{P}_{n}(\Sigma))=1$ by Lemma~\ref{lem:normalsubgroup}~(4).
Therefore $\phi(\beta)=\beta$.
\end{proof}

Recall the generating sets $X_i$ for $\mathbf{U}_i$, and their union $X=\cup_i X_i$.
We split $X_i$ into the following three sets
\[
X_i(1)=\{A_{i,j}|j\neq i\},\quad X_i(2)=\{\zeta_{i,t}|1\le t\le p\},\quad X_i(3)=\{a_{i,r}, b_{i,r}|1\le r\le g\}
\]
and let $X(k)=\bigcup_i X_i(k)$.

We define
\[
X_i^*(k)=\left\{f_*(x)|x\in X_i(k)\cup X_i^{-1}(k), f\in\mathcal{PM}_n^*(\Sigma)\right\},
\]
\[
X_i^*=\bigcup_k X_i^*(k),\quad X^*(k)=\bigcup_i X^*_i(k),\quad
X^*=\bigcup_i X_i^*.
\]

We say that a braid $\beta\in\mathbf{P}_n(\Sigma)$ is {\em of type $k$} if $\beta\in X^*(k)$, and that $\phi$ is 
\begin{enumerate}
\item {\em pre-geometric} if $\phi\left(X\right)\subset X^*$;
\item {\em almost-geometric} if $\phi\left(X_i\right) \subset X^*_j$ for some $i$ and $j$;
\item {\em type-preserving} if $\phi\left(X(k)\right)\subset X^*(k)$ for each $k$;
\item {\em geometric} if $\phi$ is type-preserving and almost-geometric.
\end{enumerate}

These definitions can be used to restate Theorem~\ref{thm:peripheral} as follows.
\begin{thm}\label{thm:peripheral2}
There exists $f\in\mathcal{M}^*_{n}(\Sigma)$ realizing $\phi$ if and only if both $\phi$ and $\phi^{-1}$ are geometric.
\end{thm}
\begin{proof}
Suppose that $f$ realizes $\phi$. Then $\phi$ is almost-geometric since $f$ transforms the peripheral structure of $\Sigma_{\hat i} = \Sigma\setminus\mathbf{z}\cup\{z_i\}$ to that of $\Sigma_{\hat j}$ for any $i$ and $j=\rho(f)(i)$.
Moreover, $\phi$ is type-preserving since $f$ preserves the geometric properties which determine the types.

Suppose that both $\phi$ and $\phi^{-1}$ are almost-geometric. Then there exist $i$ and $j$ such that $\phi(X_i)\subset X_j^*$ and $\phi^{-1}(X_j)\subset X_i^*$. Therefore $\phi(\mathbf{U}_i)=\mathbf{U}_j$.
Let $\phi_i=\phi|_{\mathbf{U}_i}$ and $\bar\phi_i:H_1(\mathbf{U}_i)\to H_1(\mathbf{U}_j)$ be the induced isomorphism.
Note that $H_1(\mathbf{U}_i)$ is a free abelian group of rank $\kappa-2$ generated by the homology classes $[x]$ for $x\in X_i$ with a single defining relator
\begin{equation}\label{eqn:abUi}\tag{HPTR}
[A_{1,i}]+\dots+[A_{i-1,i}]+[A_{i,i+1}]+\dots+[A_{i,n}]+\sum_{t=1}^p [\zeta_{i,t}]=0
\end{equation}
coming from (\ref{eqn:ptr}).

Since $\phi$ is type-preserving, $\bar\phi_i([A_{i,\ell}])=\pm[A_{j, \rho(\ell)}]$ and $\bar\phi_i([\zeta_{i,t}])=\pm[\zeta_{j,\mu(t)}]$ for some $\rho(\ell)$ and $\mu(t)$.
However (\ref{eqn:abUi}) forces all signs to be the same, and actually $\rho$ and $\mu$ define permutations on $\mathbf{z}$ and $\mathbf{p}$.
Therefore $\phi_i$ transforms the peripheral structures of $\mathbf{U}_i$ to those of $\mathbf{U}_j$ without any change of types, and so by Theorem~\ref{thm:peripheral}, there exists a homeomorphism $f$ on $(\Sigma,\mathbf{z})$ realizing $\phi_i$, in other words, 
$f_*|_{\mathbf{U}_i}=\phi_i$.

Consider $\phi f_*^{-1}\in\Aut(\mathbf{P}_n(\Sigma))$. Then $\phi f_*^{-1}|_{\mathbf{U}_i}=Id_{\mathbf{U}_i}$ and so $\phi = f_*$ by Lemma~\ref{lem:AutUi}.
\end{proof}

We claim three propositions as follows. The first two are easier to prove than the last one.
\begin{prop}\label{prop:UiPres}
If $\phi$ is pre-geometric, then $\phi$ is almost-geometric.
\end{prop}

\begin{prop}\label{prop:typepres}
If $\phi$ is pre-geometric, then $\phi$ is type-preserving.
\end{prop}

\begin{prop}\label{prop:scc}
$\phi$ is pre-geometric.
\end{prop}

Hence $\phi$ is always geometric, and therefore Theorem~\ref{thm:peripheral2} implies Proposition~\ref{prop:AutPn}. These propositions will be proven later.

\subsection{Almost-geometricity and type-preserveness}\label{sec:end}
For now, assume $\phi$ is pre-geometric.

For $x\in X_i^*$, we define $end(x)$ as 
\begin{enumerate}
\item $\{z_i,z_j\}$ if $x$ is of type 1 and is a conjugate of $A_{i,j}^{\pm1}$;
\item $\{z_i, p_t\}$ if $x$ is of type 2 and is a conjugate of $\zeta_{i,t}^{\pm1}$;
\item $\{z_i\}$ if $x$ is of type 3.
\end{enumerate}

Then $end(x)$ can be viewed geometrically as follows.
Recall the surface $\Sigma_{\hat i}=\Sigma\setminus\mathbf{z}\cup\{z_i\}$.
Let $\gamma(x)\subset\Sigma_{\hat i}$ be a simple closed curve based at $z_i$ representing $x$ in $\mathbf{U}_i$, and let $\mathcal{R}(x)$ be a set of essential simple closed curves among the boundary components of a closed regular neighborhood of $\gamma(x)$.
Then the curves in $\mathcal{R}(x)$ bound a subsurface $\Sigma(x)$ of nonnegative Euler characteristic, and $end(x)=\Sigma(x)\cap(\mathbf{z}\cup\mathbf{p})$.
Note that $x$ as a mapping class is a product of positive or negative Dehn twists along curves in $\mathcal{R}(x)$. 
Hence if $\gamma(x), \gamma(y)$ are disjoint for some $x,y\in X^*$, then $x$ and $y$ commute.

\begin{lem}\label{lem:rank2subgroup}
Let $x\neq y\in X^*$. Suppose that not both $x$ and $y$ are of type $3$. 
Then any pair of conjugates of $x$ and $y$ generates a rank-$2$ free subgroup if and only if $end(x)\cap end(y)$ is a point.
\end{lem}
\begin{proof}
By definition, $end(x)=end(y)$ implies that $x$ is a conjugate of $y$ or $y^{-1}$ since not both $x$ and $y$ are of type 3.

If $end(x)$ and $end(y)$ are disjoint, then there exists $f\in\mathcal{M}^*_n(\Sigma)$ such that $f(\gamma(x))=\gamma(f_*(x))$ is disjoint from $\gamma(y)$, and so $f_*(x)$ and $y$ commute. Since $f_*(x)$ or $f^{-1}_*(y)$ is a conjugate of $x$ or $y$, respectively, we are done.

On the other hand, if $end(x)\cap end(y)=\{z_i\}$ for some $i$, then $x, y\in \mathbf{U}_i$ and they generate a rank-2 free group. Otherwise, if $end(x)\cap end(y)=\{p_t\}$ for some $t$, then we may switch the roles of $\mathbf{z}$ and $\mathbf{p}$ as follows. 
\[
x,y\in\mathbf{P}_{n}(\Sigma_{g,p})\subset\mathbf{P}_{n+p}(\Sigma_g)\supset\mathbf{P}_{p}(\Sigma_{g,n})\supset\mathbf{U}_t.
\]
Then $x$ and $y$ can be regarded as elements in $\mathbf{U}_t$ and therefore by the same argument as above, we are done.
\end{proof}

\begin{defn}
Let $G$ be a group.
We say that a subset $Y=\{y_1, \dots, y_m\}$ in $G$ is {\em a strongly free generating set} if $\{w_1^{-1}y_1w_1,\dots,w_m^{-1}y_mw_m\}$ generates a free subgroup of rank $m$ for any $w_i$'s in $G$.
\end{defn}

Note that strong freeness does not depend on a group presentation, and one of the necessary conditions for strong freeness is as follows.
\begin{lem}\label{lem:strfree}
Let $\mathbf{F}$ be a free normal subgroup of a group $G$ and $Y=\{y_1,\dots,y_m\}$ be a finite subset of $\mathbf{F}$.
If $\{[y_1],\dots,[y_m]\}$ generates $\mathbb{Z}^m$ in $G/[G,G]$, then $Y$ is a strongly free generating set.
\end{lem}
\begin{proof}
All conjugates of elements in $Y$ are lying in $\mathbf{F}$ since $\mathbf{F}$ is normal in $G$, and they generate a free subgroup of rank $m'\le m$.
However, $\{[y_1], \dots, [y_m]\}$ generates $\mathbb{Z}^m$ only if $m'=m$, and so the condition implies strongly-freeness.
\end{proof}

Apply Lemma~\ref{lem:strfree} with $G=\mathbf{P}_{n}(\Sigma)$, $\mathbf{F}=\mathbf{U}_i$ and $Y\subset X_i$.
The lemma below is the key ingredient to prove Proposition~\ref{prop:UiPres} and Proposition~\ref{prop:typepres}.

\begin{lem}\label{lem:freesubgroup}
Let $Y=\{y_1,\dots, y_m\}\subset X^*$ with $m\ge 3$.
Suppose that at most one element in $Y$ is of type $3$.
Then $Y$ is strongly free only if
$end(x)\cap end(y)\cap end(z)$ is a point for any distinct triple $\{x,y,z\}\subset Y$, and therefore, all $end(y_i)$'s share only one point either $z_i$ or $p_t$ for some $i$ or $t$.
\end{lem}
\begin{proof}
Suppose that the lemma fails for some triple $\{x,y,z\}$.
Then $end(x)\cap end(y)\cap end(z)=\emptyset$ since all $end(y_i)$'s are different  by Lemma~\ref{lem:rank2subgroup}.

If one of them, say $x$, is of type 3, then neither $y$ nor $z$ is of type 3 by the hypothesis, and $end(x)\cap end(y)$ or $end(x)\cap end(z)$ is empty. By Lemma~\ref{lem:rank2subgroup} again, this is impossible.

Therefore all $x,y$ and $z$ are of type 1 or 2, and their ends intersect pairwise exactly at one point. It means that either all of them are of type 1, or only one of them is of type 1.
Hence by definition of types, some conjugates of $x,y$ and $z$ or their inverses are precisely equal to either $A_{1,2}$, $A_{2,3}$ and $A_{3,1}$ or $\zeta_{1,t}$, $\zeta_{2,t}$ and $A_{1,2}$, for some $1\le t\le p$.

However, these two triples satisfy the relations as follows.
\begin{align*}
A_{1,2}A_{3,1}A_{2,3} = A_{2,3}A_{1,2}A_{3,1}&=A_{3,1}A_{2,3}A_{1,2}=(\sigma_1\sigma_2)^3\\
A_{1,2} \zeta_{2,t}\zeta_{1,t} = \zeta_{1,t} A_{1,2} \zeta_{2,t} &= \zeta_{2,t}\zeta_{1,t}A_{1,2} = \zeta_t\sigma_1\zeta_t\sigma_1
\end{align*}
Therefore $\{x,y,z\}$ is not a strongly free generating set and this is a contradiction.

The last statement follows easily by varying $z$ in $Y$.
\end{proof}

Notice that the relations in the above proof are essentially same as the {\em lantern relation} in the mapping class group \cite[Proposition~5.1]{FM}.

Now we prove Proposition~\ref{prop:UiPres} and Proposition~\ref{prop:typepres}, and separate the cases according to $g$. 

\begin{proof}[Proof of Proposition~\ref{prop:UiPres} for $g=0$]
Suppose $g=0$. Then $X(3)=\emptyset$ and $p>3$ since $\chi<-1$.

We first fix $i$. 
Then for any proper subset $Y\subsetneq X_i$, we have $Y\subset \mathbf{U}_i\lhd\mathbf{P}_{n}(\Sigma)$ and $Y$ satisfies the condition on homology classes described in Lemma~\ref{lem:strfree}. Hence $Y$ is a strongly free generating set, and so is $\phi(Y)$. Hence Lemma~\ref{lem:freesubgroup} implies that $\bigcap_{y\in\phi(Y)} end(y)$ share a common point either $z_j$ or $p_t$ for any $Y\subsetneq X_i$, and for $Y=X_i$ itself.
Therefore, it suffices to show that $\bigcap_{y\in\phi(X_i)} end(y)=\{z_j\}$ for some $j$.

If $p_t\in end(\phi(y))$ for all $y\in X_i$, then there are only $n$ possibilities for $end(\phi(y))$. However since $|X_i|=p+n-1>n+1$, this is a contradiction by choosing a strongly free subset $Y\subsetneq X_i$ with $|Y|>n$.
\end{proof}

\begin{proof}[Proof of Proposition~\ref{prop:typepres} for $g=0$]
Suppose $g=0$. Since $X(3)=\emptyset$, it suffices to show that $\phi(X(1))\subset X^*(1)$.
Then by Proposition~\ref{prop:UiPres}, $\phi(X_i)\subset X^*_{i'}$ and $\phi(X_j)\subset X^*_{j'}$ for some $i'$ and $j'$.
Since $X_i\cap X_j = \{A_{i,j}\}$, 
\[
\phi(A_{i,j}) = \phi(X_i)\cap\phi(X_j)\subset X^*_{i'}\cap X^*_{j'}.
\]
Therefore $\phi(A_{i,j})$ is of type 1 since both $z_{i'}$ and $z_{j'}$ are in $end(\phi(A_{i,j}))$.
\end{proof}

\begin{proof}[Proof of Proposition~\ref{prop:typepres} for $g\ge 1$]
Suppose $g\ge 1$. Then $\phi$ preserves elements of type 1 because
the definiting relator (\ref{eqn:pscr}) implies that all $A_{i,j}$'s are in the commutator subgroup which is characteristic, and furthermore, all other types of generators survive in $H_1(\mathbf{P}_{n}(\Sigma))$ under the abelianization.

Moreover if $p=0$ then $X(2)=\emptyset$ and therefore $\phi$ is type-preserving.
Otherwise if $p>0$, we suppose that $x$ is of type 3 but $\phi(x)$ is of type 2.
Then, by (\ref{eqn:pscr}) again, there exists $y$ of type 3 such that $[x,y]$ is of type 1, and therefore so is $[\phi(x),\phi(y)]$.

Let $q:\mathbf{P}_{n}(\Sigma)\to
\mathbf{P}_{n}(\Sigma)/\langle\!\langle X(2)\rangle\!\rangle\simeq\mathbf{P}_n(\Sigma_g)$ be the quotient map coming from Corollary~\ref{cor:quotient}.
Then $q(\phi(x))=e$ and $\ker(q)\cap X^*(1)=\emptyset$. Hence
\[
e\neq q(\phi([x,y]))=q([\phi(x),\phi(y)])=[q(\phi(x)),q(\phi(y))]=e.
\]
This contradiction completes the proof.
\end{proof}

Recall Theorem~\ref{thm:Goldberg} due to Goldberg. 
The map $\iota_*:\mathbf{P}_{n}(\Sigma)\to\prod^n\pi_1\Sigma$ is nothing but the quotient map by elements of type 1.
Hence by Proposition~\ref{prop:typepres} for $g\ge 1$ above, there exists an induced isomorphism $\hat\phi$ that makes the following diagram commutative.
\[
\xymatrix{
\mathbf{P}_{n}(\Sigma)\ar[r]^\phi\ar[d]_{\iota_*}&\mathbf{P}_{n}(\Sigma)\ar[d]^{\iota_*}\\
\prod_{i=1}^n\pi_1(\Sigma,z_i)\ar[r]^{\hat\phi}&\prod_{i=1}^n\pi_1(\Sigma,z_i)
}
\]

\begin{proof}[Proof of Proposition~\ref{prop:UiPres} for $g\ge1$]
We claim that for each $i$, there exists $j$ such that
\[
\hat\phi(\pi_1(\Sigma, z_i)) = \pi_1(\Sigma, z_j).
\]

Let $x,y\in X_i(2)\cup X_i(3)$ with $end(x)\neq end(y)$.
Note that $\iota_*(x)$ and $\iota_*(y)$ do not commute because $\pi_1(\Sigma,z_i)$ is a one-relator group with $2g+p\ge 3$ generators and so they generate a free group of rank 2. However, if $\hat\phi(x)\in\pi_1(\Sigma,z_j)$ and $\hat\phi(y)\in\pi_1(\Sigma,z_k)$ for some $j\neq k$, then
this is a contradiction since $\hat\phi(\pi(x))$ and $\hat\phi(\pi(y))$ commute. Therefore $j=k$ and so the claim is proved.

By this claim, for each $i$ there exists $j$ such that $\phi(X_i(2)\cup X_i(3)) \subset X_j^*(2)\cup X_j^*(3)$. Hence 
it only remains to prove that $\phi(X_i(1))\in X_j^*(1)$.

Let $x\in X_i(1)$ and $y\in X_i(3)$. Then $\phi(x)$ is of type 1 and $\phi(y)$ is of type 3 by Proposition~\ref{prop:typepres} for $g\ge 1$ above.
Moreover, $\{x,y\}$ is strongly free by Lemma~\ref{lem:rank2subgroup} and so is $\{\phi(x), \phi(y)\}$.
Hence $end(\phi(x))$ and $end(\phi(y))$ intersect by Lemma~\ref{lem:rank2subgroup} again. This implies that $z_j\in end(\phi(x))$, and therefore $\phi(x)\in X_j^*$.
\end{proof}

\subsection{Nielsen-Thurston theory and pre-geometricity}
We first briefly review the Nielsen-Thurston classification to prove the propositions above.

Let $\beta\in\mathbf{P}_{n}(\Sigma) \subset \mathcal{M}_{n}^*(\Sigma)$.
Then $\beta$ is either {\em periodic} if it is of finite order; {\em reducible} if it is nontrivial and fixes a non-empty collection of isotopy classes of essential, pairwise disjoint, simple closed curves in $\Sigma\setminus\mathbf{z}$, called a {\em reduction system $\mathcal{R}(\beta)$}; or {\em pseudo-Anosov} if $\beta$ is neither periodic nor reducible.
In our case, since $\mathbf{P}_{n}(\Sigma)$ is torsion-free, $\beta$ is either reducible or pseudo-Anosov.

The {\em canonical reduction system} $\mathcal{R}(\beta)$
is defined as follows.
An essential simple closed curve $c$ is in $\mathcal{R}(\beta)$ if
\begin{enumerate}
\item $\{\beta^i(c)|i\in\mathbb{Z}\}$ is a reduction system;
\item $\beta^k(b)\neq b$ for any $k\neq0$ and for any essential simple closed curve $b$ with geometric intersection number $i(b,c)\neq0$.
\end{enumerate}

For simplicity, we denote the centralizer $C_{\mathbf{P}_{n}(\Sigma)}(\beta)$ by $C(\beta)$, the center $Z(C(\beta))$ by $Z(\beta)$ (this is a free abelian group), and the rank of $Z(\beta)$ by $rk(\beta)$.
It is obvious that these notions are preserved by $\phi$. 

\begin{thm}\cite[Corollary~3]{Mc}\label{thm:pAcenter}
Let $\beta\in\mathbf{P}_{n}(\Sigma)$ be a pseudo-Anosov braid. Then $rk(\beta)=1$ and moreover,
\[C(\beta)=Z(\beta)\simeq\mathbb{Z}.\] 
\end{thm}
In \cite{Mc}, Corollary~3 says that every torsion-free subgroup of the centralizer of a pseudo-Anosov mapping class in $\mathcal{M}^*_{n}(\Sigma)$ is infinite cyclic. Since $\mathbf{P}_{n}(\Sigma)$ can be identified with a subgroup of $\mathcal{M}^*_{n}(\Sigma)$ and is torsion-free, the assertion follows.

In general, let $pa(\beta)$ be the number of pseudo-Anosov restrictions of $\beta\in\mathbf{P}_{n}(\Sigma)$ on connected components of $\Sigma\setminus\mathbf{z}\setminus\mathcal{R}(\beta)$ and let $\mathcal{R}^*(\beta)$ and $\mathcal{R}^0(\beta)$ be the sets of curves which are essential and inessential, respectively, in $\Sigma$.
For each $c\in\mathcal{R}^*(\beta)$, let $[c]$ be the set of curves in $\mathcal{R}^*(\beta)$ which are isotopic to $c$ not in $\Sigma\setminus\mathbf{z}$ but in $\Sigma$. Then there exists an integer $k$ and $\{c_1,\dots,c_k\}\subset\mathcal{R}^*(\beta)$ such that $\mathcal{R}^*(\beta)$ is a disjoint union of $[c_i]$'s.
Note that since $\beta$ is isotopic to the identity in $\mathcal{M}^*(\Sigma)$, each $[c_i]$ consists of at least 2 curves.

\begin{prop}\cite[Proposition~4.1]{IIM}\label{prop:rank}
Let $\beta\in\mathbf{P}_{n}(\Sigma)$. Then 
\[
rk(\beta)=pa(\beta)+ |\mathcal{R}^0(\beta)|+\sum_{i=1}^k (|[c_i]|-1).
\]
\end{prop}
A proof is given since in \cite{IIM} the proposition is stated without a proof and this generalizes to non-closed cases.
\begin{proof}We use induction on $\chi(\Sigma)$.

Let $\Sigma\setminus (c_1\cup\dots\cup c_k)=\{\Sigma^1,\dots, \Sigma^m\}$. Then $\chi(\Sigma^j)>\chi(\Sigma)$ and we suppose that the formula above holds for each $\Sigma^j$.

Since $\chi(\Sigma^j)<0$ and any element in $C(\beta)$  as a mapping class preserves $\mathcal{R}(\beta)$ and $\mathcal{R}^*(\beta)$ as well, $C(\beta)$ can be factored as a direct product $\prod_{j=1}^m C(\beta_j)$, and therefore
\[
rk(\beta)=\sum_{j=1}^m rk(\beta_j)
\]
where $\beta_j$ is the restriction of $\beta$ on $\Sigma^j$.
Then the curves $c_i$'s are inessential in each $\Sigma^j$, and so all curves in $[c_i]$ except $c_i$ itself contribute to curves in $\mathcal{R}^0(\beta_j)$ for some $j$. Moreover, $\mathcal{R}^*(\beta_j)$ for each $j$ is empty.

Let ${\mathcal{R}^{0}}'(\beta_j)=\mathcal{R}^0(\beta)\cap\mathcal{R}^0(\beta_j)$, which consists of curves not isotopic to any $c_i$'s.
Then $\cup_{j=1}^m{\mathcal{R}^0}'(\beta) = \mathcal{R}^0(\beta)$ and so
\begin{align*}
rk(\beta)=\sum_{j=1}^m \left( pa(\beta_j) + \left|\mathcal{R}^0(\beta_j)\right| \right)= pa(\beta)+\sum_{j=1}^m\left|{\mathcal{R}^0}'(\beta_j)\right| +\sum_{i=1}^k(|[c_i]|-1).
\end{align*}

Hence we may assume that $\mathcal{R}^*(\beta)=\emptyset$. Call $d\in\mathcal{R}^0(\beta)$ {\em outermost} if there is no $d'\in\mathcal{R}^0(\beta)$ bounding a disc or once-punctured disc containing $d$. This notion is well-defined since $\chi(\Sigma)<0$.
Let $\{d_1,\dots,d_\ell\}$ be the set of outermost curves in $\mathcal{R}^0(\beta)$, 

Suppose $\Sigma\setminus(d_1\cup\dots\cup d_\ell)=\{\Sigma^0,\Sigma^1,\dots,\Sigma^\ell\}$, where $\Sigma^0<0$ and $\Sigma^j>0$ for all $j>0$.
Then as before, $C(\beta)=\prod_{j=0}^\ell C(\beta_j)$ and so $rk(\beta)=\sum_{j=0}^\ell rk(\beta_j)$, where $\beta_j$ is the restriction of $\beta$ on $\Sigma^j$.

Moreover, $\mathcal{R}(\beta_0)=\emptyset$, and $\cup_{j=1}^m\mathcal{R}^0(\beta_j)=\mathcal{R}^0(\beta)\setminus\{d_1,\dots,d_\ell\}$.
Then from the exact sequence of Theorem~1.1 in \cite{GW}, it can be easily obtained that $rk(\gamma)=pa(\gamma)+\left|\mathcal{R}^0(\gamma)\right|+1$ if $\Sigma$ is either a disc or once-punctured disc.
Hence
\begin{align*}
rk(\beta)=\sum_{j=1}^\ell rk(\beta_j)&=\sum_{j=1}^\ell \left(pa(\beta_j)+\left|\mathcal{R}^0(\beta_j)\right| + 1\right)=pa(\beta)+\left(\left|\mathcal{R}^0(\beta)\right|-\ell\right)+\ell.
\end{align*}

Therefore the above formula holds.
\end{proof}

As a consequence of these results, we have a direct indecomposability result for braid groups as follows.
Two groups $H_1$ and $H_2$ are {\em abstractly commensurable} if there are finite index subgroups $K_i\subset H_i$ for $i=1,2$ which are isomorphic.

\begin{prop}\label{prop:indecomposable}
Any group which is abstractly commensurable with $\mathbf{B}_{n}(\Sigma)$ is directly indecomposable.
\end{prop}
\begin{proof}
We first claim that any finite index subgroup of $\mathbf{B}_{n}(\Sigma)$ is directly indecomposable.

Let $K$ be a finite index subgroup of $\mathbf{B}_{n}(\Sigma)$ which is isomorphic to $K_1\times K_2$ for some nontrivial subgroups $K_1$ and $K_2$ of $K$.
Note that $\mathbf{B}_{n}(\Sigma)$ itself and any finite index subgroup of $\mathbf{B}_{n}(\Sigma)$
contain at least one pseudo-Anosov element $\beta$ since they are non-abelian \cite[Lemma~2.5]{Lo}. Hence there exist $\beta_i\in K_i$ such that they commute and $\beta=\beta_1\beta_2$.

If neither $\beta_1$ nor $\beta_2$ is trivial, then $C(\beta)\cap K$ contains a $\mathbb{Z}^2$ subgroup generated by $\beta_1$ and $\beta_2$. However $C(\beta)=\mathbb{Z}$ by Theorem~\ref{thm:pAcenter}, and so this is a contradiction.

If one of $\beta_i$ is trivial, then either $K_1$ or $K_2$ is contained in $C(\beta)\cap K\subset C(\beta)=\mathbb{Z}$ and it is abelian. This means that the center $Z(K)$ is nontrivial. However, this is contradiction because $K$ is centerless by Lemma~\ref{lem:centerBraid}~(3). Therefore $K$ is directly indecomposable and the claim is proved.

Now let $H$ be any group containing $K$ as a finite index subgroup. Suppose $H=H_1\times H_2$ as above. Then $K\cap H_i$ is a nontrivial finite index subgroup of  $H_i$ for each $i$, and therefore $(K\cap H_1)\times(K\cap H_2)$ is a finite index subgroup of $H$, and so of $K$, and furthermore of $\mathbf{B}_{n}(\Sigma)$. This is contradiction because any finite index subgroup of $\mathbf{B}_{n}(\Sigma)$ is directly indecomposable.
\end{proof}

Let $x\in X_i^*(k)$. Since the rank formula is invariant under taking conjugates and inverses, $rk(x)$ is equal to $rk(g)$ for some $g\in X_i(k)$.
Then it follows easily from the definition that $g$ is reducible and $rk(g)=1$. Indeed, the canonical reduction system $\mathcal{R}(g)$ is exactly same as the set defined in Section~\ref{sec:end}.

Conversely, suppose that $rk(\beta)=1$ for given $\beta\in\mathbf{P}_{n}(\Sigma)$.
If $\beta$ is pseudo-Anosov, then $C(\beta)=\mathbb{Z}$ by Theorem~\ref{thm:pAcenter} and there is nothing to do more.

Assume that $\beta$ is reducible.
If $pa(\beta)=1$, then $\mathcal{R}^0(\beta)=\emptyset$ and $[c_i]=\{c_i\}$ for all $1\le i\le k$. This cannot happened since, as recalled the argument just before the Proposition~\ref{prop:rank}, each $[c_i]$ consists of at least 2 curves. Therefore $pa(\beta)$ should be 0. 
Then $\beta$ is a product of nontrivial powers of Dehn twists along curves in $\mathcal{R}(\beta)$ by the Nielsen-Thurston classification, and by the same argument, either
$\mathcal{R}(\beta)=\{c\}$ where $c$ bounds a disc or once-punctured disc, or
$\mathcal{R}(\beta)=\mathcal{R}^*(\beta)=[c]=\{c,d\}$
where $c$ and $d$ are isotopic to each other in $\Sigma$.

The summary of the above discussion is as follows. 
For given reducible braid $\beta\in\mathbf{P}_n(\Sigma)$, we have $rk(\beta)=1$ if and only if either
\begin{enumerate}
\item $\mathcal{R}(\beta)=\{c\}$, where $c$ bounds a disc in $\Sigma$;
\item $\mathcal{R}(\beta)=\{c\}$, where $c$ bounds a once-punctured disc in $\Sigma$;
\item $\mathcal{R}(\beta)=\{c, d\}$, where $c\cup d$ bounds an annulus in $\Sigma$.
\end{enumerate}

We analyze each case in detail. Suppose $rk(\beta)=1$. We further assume that $[\beta]\neq0$ in $H_1(\mathbf{P}_{n}(\Sigma))$ when $\mathcal{R}(\beta)$ bounds an annulus. Then $\mathcal{R}(\beta)$ always separates $\Sigma$ into two pieces $\Sigma_+$ and $\Sigma_-$.
Moreover, we may assume that $\Sigma_+$ is either a disc, once-punctured disc, or an annulus.
Let $\overline\Sigma_+$ be the closure of $\Sigma_+$ in $\Sigma$ and $n_\pm=|\mathbf{z}\cap\Sigma_\pm|$.

Since $C(\beta)$ preserves $\mathcal{R}(\beta)$, we can consider the map $s_-:C(\beta)\to\mathcal{M}^*_{n_-}(\Sigma_-)$ induced from the restriction to $\Sigma_-$.
This is well-defined since the only possible ambiguities come from Dehn twists along $\mathcal{R}(\beta)$ but they are already trivial in $\mathcal{M}^*_{n_-}(\Sigma_-)$.
Notice that $s_-(\beta)$ is trivial since $\beta$ has no pseudo-Anosov components and the curves in $\mathcal{R}(\beta)$ are not essential in $\Sigma_-$. 
Then the map 
\[
i_-:\mathcal{M}^*_{n_-}(\overline\Sigma_-)\to C(\beta)\subset\mathcal{M}^*_n(\Sigma)
\]
induced from the embedding $(\Sigma_-,\mathbf{z}\cap\Sigma_-)\to(\Sigma,\mathbf{z})$ is injective \cite[Theorem~4.1]{PR2}.

On the other hand, the embedding $\Sigma_+\to\Sigma$ induce $i_+:\mathbf{P}_{n_+}(\Sigma_+)\to\mathbf{P}_{n}(\Sigma)$, which is injective by \cite[Theorem~2.3]{PR1}. 
Moreover, the image of $Z(\mathbf{P}_{n_+}(\Sigma_+))$ under $i_+$ is the same as $Z(\beta)$, and so $\beta$ can be regarded as a central element $\mathbf{P}_{n_+}(\Sigma_+)$.
Therefore via $i_+$, $C(\beta)$ contains all of $\mathbf{P}_{n_+}(\Sigma_+)$ which is $\ker(s_-)$ by definition, hence there exists a short exact sequence
\[
\xymatrix{
1\ar[r]&\mathbf{P}_{n_+}(\Sigma_+)\ar[r]& C(\beta)
\ar[r]^-{s_-}&\operatorname{im}(s_-)\ar[r]&1.
}
\]

In the first case, $\Sigma_+\simeq\mathbb{R}^2$ and $\Sigma_-\simeq\Sigma\setminus\{*\}$, where a point $*$ plays the same role of $\overline\Sigma_+$ in $\Sigma$.
For any $\alpha\in C(\beta)\subset\mathbf{P}_n(\Sigma)$, the element $s_-(\alpha)$ in $\mathcal{M}^*_{n_-}(\Sigma_-)$ is obtained by forgetting all marked points in $\Sigma_+$ except only one and can be regarded as an element in $\mathcal{M}^*_{n_-+1}(\Sigma)$. Moreover, by forgetting all remaining marked points, we have a trivial mapping class. This means that the image $s_-(\alpha)$ is lying in $\mathbf{P}_{n_-+1}(\Sigma)\subset\mathcal{M}^*_{n_-}(\Sigma_-)$ by Birman exact sequence.
Since $\mathcal{M}^*_{n_-}(\Sigma_-)\simeq\mathcal{M}^*_{n_-}(\overline\Sigma_-)/\langle t_c\rangle$, where $t_c$ is the Dehn twist along $c$, we have a commutative diagram
\[
\xymatrix{
\mathcal{M}^*_{n_-}(\overline\Sigma_-)\ar@{->}[r]^-{i_-}\ar[d] & C(\beta)\ar[d]_{s_-}\\
\mathcal{M}^*_{n_-}(\Sigma_-)\ar@/^1pc/@{..>}[u]^r &\mathbf{P}_{n_-+1}(\Sigma)\ar@{->}[l]\ar@/_1pc/@{..>}[u]_{\bar r},
}
\]
where all horizontal maps are injective.

Let us fix a function $r$ which is a right inverse of the vertical map.
Then for any $\alpha\in\mathbf{P}_{n_-+1}(\Sigma)$, the image $i_-(r(\alpha))$ commutes with $\beta$ in $\mathcal{M}^*_n(\Sigma)$ since the support of $i_-(r(\alpha))$ is contained in $\Sigma_-$.
Therefore $r$ induces a right inverse $\bar r:\mathbf{P}_{n_-+1}(\Sigma)\to C(\beta)$ of $s_-$ and so $\operatorname{im}(s_-)=\mathbf{P}_{n_-}(\Sigma)$.

However, $r$ and $\bar r$ are not homomorphisms in general. Indeed, the only possible obstruction for $r$ and $\bar r$ to be splitting maps is the element $t_c\in C(\beta)$.
Since $t_c$ correspond to $\Delta^2\in Z(\mathbf{P}_{n_+}(\mathbb{R}^2))\simeq Z(\beta)$ via $i_+$, the short exact sequence 
\[
\xymatrix{
1\ar[r]& \mathbf{P}_{n_+}(\mathbb{R}^2)_{/Z}\ar[r]& C(\beta)_{/Z}\ar[r]^-{s_-}& \mathbf{P}_{n_-+1}(\Sigma)\ar[r]& 1
}
\]
splits now and so
\[
C(\beta)_{/Z}\simeq\mathbf{P}_{n_+}(\mathbb{R}^2)_{/Z}\times\mathbf{P}_{n_-+1}(\Sigma).
\]

In the second case, $\Sigma_+\simeq A$ and $\Sigma_-\simeq \Sigma$.
Then the puncture in $\Sigma_+$ implies a kind of rigidity of $\overline\Sigma_+$ in $\Sigma$, and now $\overline\Sigma_+$ acts like a puncture, not a marked point as above.
We have $\operatorname{im}(s_-)\simeq\mathbf{P}_{n_-}(\Sigma)$ by the similar argument.
Since 
\[
\xymatrix{
\mathbf{P}_{n_-}(\Sigma)\simeq\mathbf{P}_{n_-}(\overline\Sigma_-)\subset\mathcal{M}^*_{n_-}(\overline\Sigma_-)\ar[r]^-{i_-} &C(\beta),
}
\]
the sequence 
\[
\xymatrix{
1\ar[r]&\mathbf{P}_{n_+}(A)\ar[r]&C(\beta)\ar[r]^-{s_-}&\mathbf{P}_{n_-}(\Sigma)\ar[r]&1
}
\]
splits and therefore
\[
C(\beta)\simeq\mathbf{P}_{n_+}(A)\times\mathbf{P}_{n_-}(\Sigma),\quad C(\beta)_{/Z}\simeq\mathbf{P}_{n_+}(A)_{/Z}\times\mathbf{P}_{n_-}(\Sigma).
\]

In the last case, $\Sigma_+\simeq A$ and $\Sigma_-$ is of genus $g-1$ with two more punctures.
Then $\operatorname{im}(s_-)\simeq\mathbf{P}_{n_-}(\Sigma_-)$ because two essential curves in $\mathcal{R}(\beta)$ act like two punctures in $\Sigma_-$. 
As before, we have
\[
\xymatrix{
\mathbf{P}_{n_-}(\Sigma_{g-1,p+2})\simeq\mathbf{P}_{n_-}(\overline\Sigma_-)\subset\mathcal{M}^*_{n_-}(\overline\Sigma_-)\ar[r]^-{i_-} &C(\beta),
}
\]
which is a splitting map of $s_-$. Therefore
\[
C(\beta)\simeq\mathbf{P}_{n_+}(A)\times\mathbf{P}_{n_-}(\Sigma_{g-1,p+2}), \quad C(\beta)_{/Z}\simeq\mathbf{P}_{n_+}(A)_{/Z}\times\mathbf{P}_{n_-}(\Sigma_{g-1,p+2}).
\]

\begin{rmk}\label{rmk:homnontrivial}
If we drop the nontriviality condition for $[\beta]\in H_1(\mathbf{P}_n(\Sigma))$, then $\mathcal{R}(\beta)$ decomposes $\Sigma$ into 3 pieces $\Sigma_+, \Sigma_-^1$ and $\Sigma_-^2$, where $\Sigma_-^i$ both have negative Euler characteristic. By exactly the same argument, one can show that
\[
C(\beta)\simeq\mathbf{P}_{n_+}(A)\times\mathbf{P}_{n_-^1}(\Sigma_-^1)\times\mathbf{P}_{n_-^2}(\Sigma_-^2).
\]
Therefore $C(\beta)$ as well as $C(\beta)_{/Z}$ are directly decomposable.
\end{rmk}

\begin{cor}\label{cor:centralizer}
Let $x\in X^*$. Then
$Z(x)=\langle x\rangle$
and
\[
C(x)\simeq\begin{cases}
Z(x)\text{-by-}\left(\mathbf{P}_{n-1}(\Sigma)\right)& x\in X^*(1);\\
Z(x)\times\mathbf{P}_{n-1}(\Sigma)& x\in X^*(2);\\
Z(x)\times\mathbf{P}_{n-1}(\Sigma_{g-1,p+2})& x\in X^*(3).
\end{cases}
\]
Therefore, $C(x)_{/Z}$ is directly indecomposable.
\end{cor}
\begin{proof}
This follows easily from the discussion above. Note that $n_+=2$ if $x\in X^*(1)$ and $n_+=1$ otherwise, and so in all cases $\mathbf{P}_{n_+}(\Sigma_+)$ is an infinite cyclic and isomorphic to $Z(x)$.

The direct indecomposability for $C(x)_{/Z}$ follows from Proposition~\ref{prop:indecomposable} since $C(x)_{/Z}$ is a pure braid group of some surface with $\chi<0$.
\end{proof}

\begin{lem}\label{lem:converse}
Let $\beta\in\mathbf{P}_{n}(\Sigma)$. Suppose that $rk(\beta)=1$ and $C(\beta)_{/Z}$ is nontrivial and directly indecomposable. Then either $C(\beta)\simeq \mathbf{P}_{n}(A)$ or $\beta\in X^*$, where $\mathbf{P}_{n}(A)$ is identified with a subgroup of $\mathbf{P}_{n}(\Sigma)$ via the injection induced from the embedding $A\to\Sigma$.
\end{lem}
\begin{proof} 
By the discussion above and Remark~\ref{rmk:homnontrivial}, for some $n_\pm$, $C(\beta)_{/Z}$ is either
\[
\mathbf{P}_{n_+}(\mathbb{R}^2)_{/Z}\times\mathbf{P}_{n_-+1}(\Sigma),\quad
\mathbf{P}_{n_+}(A)_{/Z}\times\mathbf{P}_{n_-}(\Sigma),\text{ or }
\]
\[
\mathbf{P}_{n_+}(A)_{/Z}\times\mathbf{P}_{n_-}(\Sigma_{g-1,p+2}),
\]
and moreover, one of two factors must be trivial in each case.

If $\mathbf{P}_{n_+}(\mathbb{R}^2)_{/Z}$ or $\mathbf{P}_{n_+}(A)$ is trivial, then $n_+=2$ or $n_+=1$, respectively. This happens only if $\beta\in X^*$.

Otherwise, the only possibilities are the second and third cases with $n_-=0$
since $\mathbf{P}_{n_-+1}(\Sigma)$ is never trivial.
Hence $n_+=n$, and $C(\beta)\simeq\mathbf{P}_{n}(A)$ via the injection $i_+$.
\end{proof}

Now we are ready to prove Proposition~\ref{prop:scc}.

\begin{proof}[Proof of Proposition~\ref{prop:scc}]
Let $x\in X$ and let $\beta=\phi(x)$. Then $rk(x)=rk(\beta)=1$ and $C(x)_{/Z}$ is nontrivial and directly indecomposable by Corollary~\ref{cor:centralizer}, as is $C(\beta)_{/Z}$. Therefore either $C(\beta)\simeq \mathbf{P}_{n}(A)$ or $\beta\in X^*$ by Lemma~\ref{lem:converse}.

If $C(\beta)\simeq \mathbf{P}_{n}(A)$, then 
\begin{align*}
\mathbf{P}_{n-1}(\Sigma_{0,3})&\simeq C(\beta)_{/Z}\stackrel{\phi}{\simeq}
\begin{cases}
\mathbf{P}_{n-1}(\Sigma) & x\in X(1)\cup X(2);\\
\mathbf{P}_{n-1}(\Sigma_{g-1,p+2}) & x\in X(3).
\end{cases}
\end{align*}
This implies that $(g,p)$ is either $(0,3)$ or $(1,1)$, which are excluded by the hypothesis $\chi<-1$. Therefore $\beta\in X^*$.
\end{proof}

\section{Proof of Theorem~\ref{thm:Characteristic}}
Recall Theorem~\ref{thm:Characteristic} that $\mathbf{P}_n(\Sigma)$ is a characteristic subgroup of $\mathbf{B}_n(\Sigma)$ except for $(\Sigma,n)=(A,2)$.
Let us review the known results briefly. Since $\mathbf{P}_1(\Sigma)=\mathbf{B}_1(\Sigma)$, there is nothing to prove for $n=1$.

\begin{thm}\cite{Iv1}\label{thm:PnCharacteristic}
Let $n\neq2$ and $\Sigma$ be any (possibly non-orientable) surface of finite type. If $n=4$, assume furthermore that $\Sigma$ does not embed in $S^2$. Then $\mathbf{P}_n(\Sigma)$ is a characteristic subgroup of $\mathbf{B}_n(\Sigma)$.
\end{thm}
The above theorem covers all but a few exceptional cases.
Indeed, if $g>0$ and $n>2$ then $\mathbf{P}_{n}(\Sigma)$ is characteristic by the above theorem.

\begin{lem}\label{lem:2char}
The pure $2$-braid group $\mathbf{P}_2(\Sigma)$ is a characteristic subgroup of $\mathbf{B}_2(\Sigma)$ unless $\Sigma$ is a $p$-punctured sphere for some $p\ge 1$.
\end{lem}
\begin{proof}
The induced permutation $\rho:\mathbf{B}_2(\Sigma)\to \mathbf{S}_2\simeq\mathbb{Z}_2$ factors through
$H_1(\mathbf{B}_{2}(\Sigma))$ whose torsion subgroup is precisely $\mathbb{Z}_2$ unless $\Sigma$ is a $p$-punctured sphere.
This follows directly from the group presentation of $\mathbf{B}_2(\Sigma)$. See Theorem~\ref{thm:presentation} for orientable cases and \cite[Theorem~A.2, Theorem~A.3]{Bel1} for non-orientable cases.

Hence the map $\rho$ is nothing but a composition of the abelianization and projection onto the torsion factor $\mathbb{Z}_2$. Therefore any automorphism on $\mathbf{B}_2(\Sigma)$ preserves the kernel $\mathbf{P}_{2}(\Sigma)$ of $\rho$.
\end{proof}

Bellingeri showed in \cite{Bel2} that $\mathbf{P}_4(S^2)$ is characteristic, hence the cases when $p=0$ or $g>0$ are done.
In \cite{Art}, E.~Artin showed that the classical pure braid group $\mathbf{P}_{n}(\mathbb{R}^2)$ is characteristic.
Furthermore, by Bell and Margalit \cite{BM}, Charney and Crisp \cite{CC}, for $\Sigma=\Sigma_{0,p}$ with $p+n\ge 5$ and $p\le 3$, $\mathbf{P}_{n}(\Sigma)$ is characteristic because $\Aut(\mathbf{B}_{n}(\Sigma))$ can be expressed by using $\mathcal{M}^*_{n}(\Sigma)$ and $\tv(\mathbf{B}_{n}(\Sigma))$ which preserve $\mathbf{P}_{n}(\Sigma)$.

\begin{lem}\label{lem:notchar}
The pure braid group $\mathbf{P}_{2}(A)$ of the twice-punctured sphere is not a characteristic subgroup of $\mathbf{B}_{2}(A)$.
\end{lem}
\begin{proof}
Note that $\mathbf{B}_{2}(A)$ admits the following group presentation.
\[
\mathbf{B}_{2}(A)=\langle\sigma_1,\zeta_1|\zeta_1\sigma_1\zeta_1\sigma_1 = \sigma_1\zeta_1\sigma_1\zeta_1\rangle,
\]
where $\rho$ maps $\sigma_1$ to the generator for $\mathbf{S}_2$ and $\zeta_1$ to trivial.

However, since there is no difference between $\sigma_1$ and $\zeta_1$ in the presentation, one can consider the automorphism which interchanges them. It is obvious that this automorphism does not preserve the kernel of $\rho$.
\end{proof}

\begin{proof}[Proof of Theorem~\ref{thm:Characteristic}]
It suffices to prove the theorem only for $g=0, p\ge 4$ and $n\ge 2$.

Let $\Sigma=\Sigma_{0,p}$ and $\psi\in\Aut(\mathbf{B}_{n}(\Sigma))$.
To prove that $\psi(\mathbf{P}_{n}(\Sigma))\subset\mathbf{P}_{n}(\Sigma)$, we claim that $\psi(X(2))\subset X^*(2)$. 
Then by Corollary~\ref{cor:quotient}, $\psi$ induces an isomorphism
\[
\mathbf{B}_{n}(\Sigma)/\langle\!\langle \zeta_t\rangle\!\rangle\simeq \mathbf{B}_{n}(S^2).
\]
Since $\rho(\zeta_t)$ is trivial, this quotient map commutes with $\rho$. Then we can use Bellingeri's result for $\mathbf{B}_n(S^2)$ to complete the proof.

Now we prove the claim as follows.
Let $x=\zeta_{i,t}\in X(2)$ and $m\ge1$ be the smallest integer so that $\psi(x^m)\in\mathbf{P}_{n}(\Sigma)$.
Recall that $C(x)$, $Z(x)$ and $rk(x)$.
Then $C(x)=C(x^m)$ and $Z(x)=Z(x^m)=\langle x\rangle$.

Consider two subgroups $\psi(C(x^m))$ and $C(\psi(x^m))$ of $\psi(\mathbf{P}_{n}(\Sigma))$ and $\mathbf{P}_{n}(\Sigma)$
and let $C=\psi(C(x^m))\cap C(\psi(x^m))$. The first observation is that 
\[
C=\psi(C(x^m))\cap \mathbf{P}_{n}(\Sigma) = \psi(\mathbf{P}_{n}(\Sigma))\cap C(\psi(x^m))
\]
and so $C$ is a finite index subgroup of both $\psi(C(x^m))$ and $C(\psi(x^m))$.
Then the tower $Z(C)\subset C\subset \psi(C(x^m))$ induces the tower
\[
\frac{Z(C)}{\psi(Z(x^m))\cap C} \subset \frac{C}{\psi(Z(x^m))\cap C}\subset \frac{\psi(C(x^m))}{\psi(Z(x^m))}\simeq \mathbf{P}_{n-1}(\Sigma).
\]

The group in the left is contained in the center of the middle which is a finite index subgroup of the right which is isomorphic to $\mathbf{P}_{n-1}(\Sigma)$. However, the center of any finite index subgroup of $\mathbf{P}_{n-1}(\Sigma)$ is trivial by Lemma~\ref{lem:centerBraid}~(3), and so 
\[
Z(C)=\psi(Z(x^m))\cap C =\langle \psi(x^m)\rangle.
\]

On the other hand, the tower $\langle \psi(x^m)\rangle \subset Z(\psi(x^m))\cap C \subset Z(C)$ gives us $Z(\psi(x^m))\cap C = Z(C)=\langle \psi(x^m)\rangle$. Since $Z(\psi(x^m))$ contains $Z(\psi(x^m))$ as a finite index subgroup, it is infinite cyclic. Therefore $rk(\psi(x^m))=1$.

Moreover, $C(\psi(x^m))/Z(\psi(x^m))$ contains $C/Z(\psi(x^m))\cap C$ as a finite index subgroup, which is also a finite index subgroup of $\psi(C(x^m))/\psi(Z(x^m))\simeq\mathbf{P}_{n-1}(\Sigma)$. In other words, $C(\psi(x^m))/Z(\psi(x^m))$ is abstractly commensurable to $\mathbf{P}_{n-1}(\Sigma)$, and therefore it is directly indecomposable by Proposition~\ref{prop:indecomposable}.

Now we apply Lemma~\ref{lem:converse} to $\psi(x^m)$. Either $C(\psi(x^m))\simeq\mathbf{P}_{n}(A)$ or $\psi(x^m)\in X^*$.
Suppose $C(\psi(x^m))\simeq\mathbf{P}_{n}(A)$. Then $\psi(x^m)$ is central in $\mathbf{P}_n(A)$ and so in $\mathbf{B}_n(A)$ as well. Hence it is easy to show that
\begin{align*}
\langle x\rangle\times \mathbf{B}_{n-1}(\Sigma)&=
C_{\mathbf{B}_{n}(\Sigma)}(x)
\simeq \psi \left(C_{\mathbf{B}_{n}(\Sigma)}(x^m)\right)\\
&=C_{\mathbf{B}_{n}(\Sigma)}(\psi(x^m))
=\mathbf{B}_{n}(A).
\end{align*}
This is impossible because $\chi<-1$, and so $\psi(x^m)\in X^*$.

Consider the homology class $\bar\psi[x]=[\psi(x)]$ in $H_1(\mathbf{B}_{n}(\Sigma))$.
Then $m\bar\psi[x]$ is either $\ell[\zeta_u]$ or $\ell[A_{i,j}] = 2\ell[\sigma_1]$ for some $u$ and $\ell\neq 0$.
However since $[\sigma_1]$ and $[\zeta_u]$ are primitive, either $\bar\psi[x]=\pm[\zeta_u]$ when $m=1$ or $\pm[\sigma_1]$ when $m=2$.

In summary, what we have shown above is that for any $1\le t\le p$, $\bar\psi([\zeta_t])$ is either $\pm[\zeta_u]$ for some $u$ or $\pm[\sigma_1]$. 
Note that $\bar\psi([\zeta_t])=\pm[\sigma_1]$ for at most one $t$ because all $[\zeta_t]$'s are pairwise linearly independent. 
Otherwise, if there is no such $t$, then we are done.

Suppose $\bar\psi([\zeta_t])=\pm[\sigma_1]$ for some $t$.
Recall that $H_1(\mathbf{B}_{n}(\Sigma))$ is generated by the $(p+1)$ homology classes
$
\{[\sigma_1],[\zeta_1],\dots,[\zeta_p]\}
$
with a single defining relator
\begin{equation}\label{eqn:HTR}\tag{HTR}
2(n-1)[\sigma_1]+\sum_{u=1}^p[\zeta_u]=0
\end{equation}
coming from (TR).

We choose a basis $\{[\sigma_1], [\zeta_u]|u\neq t\}$ for $H_1(\mathbf{B}_{n}(\Sigma))$ and let 
\[
\bar\psi([\sigma_1])=c_0[\sigma_1]+\sum_{u\neq t} c_u[\zeta_u]
\]
 for some $c_i\in\mathbb{Z}$.
Then by applying $\bar\psi$ to the defining relator (\ref{eqn:HTR}), we have
\begin{align*}
0&=\bar\psi\left(2(n-1)[\sigma_1]+\sum_{u=1}^p[\zeta_u]\right)\\
&=2(n-1)\left(c_0[\sigma_1]+\sum_{u\neq t} c_u[\zeta_u]\right)+\left(\pm[\sigma_1]+\sum_{u\neq t}\pm[\zeta_u]\right)\\
&=(2(n-1)c_0\pm1)[\sigma_1]+\sum_{u\neq t} (c_u\pm1)[\zeta_u].
\end{align*}

The last equation can not be 0 since the coefficient of $[\sigma_1]$ is odd.
This contradiction completes the proof.
\end{proof}

The characteristicity of pure braid groups are summarized in Table~\ref{tab:char1} and Table~\ref{tab:char2}.

\begin{table}[ht]
\begin{tabular}{c|cccc}
& $n=2$ & $n=3$ & $n=4$ & $n\ge5$\\
\hline
$g=0$ & Table~\ref{tab:char2} & True (\cite{Iv1}) & Table~\ref{tab:char2} & True (\cite{Iv1})\\
$g\ge1$ & True (Lemma~\ref{lem:2char}) & \multicolumn{3}{c}{True (\cite{Iv1})} 
\end{tabular}
\caption{Characteristicity of pure braid groups}
\label{tab:char1}
\end{table}

\begin{table}[ht]
\begin{tabular}{r|cc}
 & $n=2$ & $n=4$ \\
\hline
$g=0, p=0$ & True (Lemma~\ref{lem:2char}) & True \cite{Bel2}\\
$p=1$ & \multicolumn{2}{c}{True \cite{Art}} \\
$p=2$ & False (Lemma~\ref{lem:notchar}) & True (\cite{BM,CC}) \\
$p=3$ & \multicolumn{2}{c}{True (\cite{BM, CC})} \\ 
$p\ge 4$ & \multicolumn{2}{c}{True (Theorem~\ref{thm:Characteristic})} 
\end{tabular}
\caption{Characteristicity of pure braid groups for $g=0$}
\label{tab:char2}
\end{table}

\section{Exceptional cases}
Here we compute $\Aut(\mathbf{B}_{n}(\Sigma))$ for each exceptional case.

\subsection{The torus $T$ with $n\ge 2$}
We consider the composition $\bar\Gamma\circ\Psi\circ (\cdot)_*$.
\begin{align*}
\mathcal{M}^*_n(T)
\stackrel{(\cdot)_*}{\longrightarrow}
\Aut(\mathbf{B}_n(T))
\stackrel{\Psi}{\longrightarrow}
\Aut\left(\mathbf{B}_n(T)_{/Z}\right)
\stackrel{\bar\Gamma}{\longrightarrow}\Aut\left(\mathbf{P}_n(T)_{/Z}\right).
\end{align*}

If $n\ge 3$, then this is an isomorphism by Theorem~\ref{thm:kappa5}.
Therefore $\bar\Gamma$ is surjective and so becomes an isomorphism by Lemma~\ref{lem:vertInj}. 
This implies surjectivity and splitness for $\Psi$ and so
\[
\Aut(\mathbf{B}_n(T))\simeq\tv(\mathbf{B}_n(T))\rtimes\mathcal{M}^*_n(T)\simeq
GL(2,\mathbb{Z})[n]\rtimes\mathcal{M}^*_n(T).
\]

Suppose $n=2$. Then $\mathbf{P}_2(T)_{/Z}\simeq \mathbf{F}_2$ and $\Aut(\mathbf{F}_2)\simeq\mathcal{M}^*_2(T)_{/Z}$.
Hence the composition $\bar\Gamma\circ\Psi\circ (\cdot)_*$ is surjective and so $\bar\Gamma$ becomes an isomorphism as well. Furthermore $\Psi$ is surjective and splits since $\mathcal{M}^*_2(T)=\mathcal{M}^*_2(T)_{/Z}\times Z(\mathcal{M}^*_2(T))$.
Therefore, 
\begin{align*}
\Aut(\mathbf{B}_2(T))&\simeq\tv(\mathbf{B}_2(T))\rtimes\Aut(\mathbf{B}_2(T)_{/Z})
\simeq GL(2,\mathbb{Z})[2]\rtimes\mathcal{M}^*_2(T)_{/Z}.
\end{align*}

As seen in Lemma~\ref{lem:intersection}, the center $Z(\mathcal{M}^*_2(T))\simeq\mathbb{Z}/2\mathbb{Z}$, which is generated by the hyperelliptic involution $s$ interchanging two marked points, corresponds to $-I\in GL(2,\mathbb{Z})[2]\simeq\tv(\mathbf{B}_2(T))$. Note that Theorem~3.1 in \cite{Zh2} missed this point for $n=2$.

\subsection{The once-punctured torus $\Sigma_{1,1}$ with $n\ge 2$}
Since $\mathbf{B}_{n}(\Sigma_{1,1})$ is centerless, the composition $\Gamma\circ(\cdot)_*$ gives
\[
\mathcal{M}^*_{n}(\Sigma_{1,1})\stackrel{(\cdot)_*}{\longrightarrow}
\Aut(\mathbf{B}_{n}(\Sigma_{1,1}))\stackrel{\Gamma}{\longrightarrow}\Aut(\mathbf{P}_{n}(\Sigma_{1,1}))\simeq\mathcal{M}^*_{n+1}(T).
\]

Notice that 
\[
\mathbf{B}_{n}(\Sigma_{1,1})=\ker\left(\pi:\rho^{-1}(\mathbf{S}_1\times \mathbf{S}_n)\to\pi_1(T)\right),
\]
where $\mathbf{S}_1\times \mathbf{S}_n$ is a subgroup of $\mathbf{S}_{n+1}$ fixing the first letter, the map $\rho:\mathbf{B}_{n+1}(T)\to\mathbf{S}_{n+1}$ is the induced permutation, and the map $\pi$ forgets all strands but the first one.
On the other hand $f\in\mathcal{M}^*_{n+1}(T)$ induces $f_*\in\Aut(\mathbf{B}_{n+1}(T))$.

Hence $\Aut(\mathbf{B}_{n}(\Sigma_{1,1}))$ consists of $f\in\mathcal{M}^*_{n+1}(T)$ such that $f_*(\mathbf{B}_{n}(\Sigma_{1,1}))=\mathbf{B}_{n}(\Sigma_{1,1})$. In other words, the conjugation by $\rho(f)$ preserves $\mathbf{S}_n\times\mathbf{S}_1$.
This implies that $\rho(f)$ itself is in $\mathbf{S}_1\times\mathbf{S}_n$ and so $f\in\mathcal{M}^*_{n}(\Sigma_{1,1})$.
Therefore $\Aut(\mathbf{B}_{n}(\Sigma_{1,1}))\simeq\mathcal{M}^*_{n}(\Sigma_{1,1})$.

\subsection{The twice-punctured sphere $A$ with $n=2$}
\label{sec:exceptional}
Let us fix a group presentation for $\mathbf{B}_{2}(A)$ as follows.
\[\mathbf{B}_{2}(A)=\langle \sigma_1, \zeta_1|[\zeta_1,\sigma_1\zeta_1\sigma_1]\rangle
=\langle\delta,\zeta|[\zeta,\delta^2]\rangle,
\]
where $\zeta=\zeta_1$ and $\delta=\zeta_1\sigma_1$. 
Then $\mathbf{B}_2(A)$ is isomorphic to the Artin group $I_2(4)$, or Baumslag-Solitar group $BS(2,2)$ of type $(2,2)$, which have been already studied in \cite{BM, GHMR}.
It is known from \cite[Proposition~9]{BM} that
\[
\Aut(\mathbf{B}_2(A))=\tv(\mathbf{B}_2(A))\rtimes\Aut^*(\mathbf{B}_2(A)),
\]
where $\Aut^*(\mathbf{B}_2(A))$ is the subgroup of length-preserving or reversing automorphisms with respect to the length function
\[
\ell:\mathbf{B}_2(A)\to\mathbb{Z},\quad\ell(\sigma_1)=\ell(\zeta_1)=1.
\]
In other words, $\phi\in\Aut^*(\mathbf{B}_2(A))$ if and only if $\ell\circ\phi=\pm\ell$, and so it is obvious that $\Inn(\mathbf{B}_2(A))\subset\Aut^*(\mathbf{B}_2(A))$.

On the other hand, it is known from \cite[Theorem~D]{GHMR} that
\begin{align*}
\Out(\mathbf{B}_2(A))&\simeq D_\infty\times \mathbb{Z}/2\mathbb{Z}=\langle \pi, s_*| \pi^2,  s_*^2\rangle\times\langle\tau_*|\tau_*^2\rangle,
\end{align*}
where $D_\infty$ is the infinite dihedral group and
\[
s_*:
\begin{cases}
\delta\mapsto\delta^{-1}\\
\zeta\mapsto\zeta \delta^{-2},
\end{cases}\quad
\tau_*:
\begin{cases}
\delta\mapsto\delta^{-1}\\
\zeta\mapsto\zeta^{-1},
\end{cases}\quad
\pi:
\begin{cases}
\delta\mapsto\delta^{-1}\\
\zeta\mapsto\zeta \delta^{-1}.
\end{cases}
\]
The automorphisms $s_*$ and $\tau_*$ are induced from the hyperelliptic involution $s$ and the orientation-reversing involution $\tau$ in $\mathcal{M}^*_2(A)$ so that 
\[
\mathcal{M}^*(A)=\langle s|s^2\rangle\times\langle \tau|\tau^2\rangle,\quad Z(\mathcal{M}^*_2(A))=\langle s|s^2\rangle,
\]
and both $\tau_*$ and $\pi$ are lying in $\Aut^*(\mathbf{B}_2(A))$ but $s_*$ is not.

Let us denote by $\zeta_*$ and $\delta_*$ the inner automorphisms defined as conjugation by $\zeta$ and $\delta$. Then 
\[
\Inn(\mathbf{B}_2(A))=\langle\zeta_*\rangle\ast\langle\delta_*|\delta_*^2\rangle\simeq\mathbf{B}_2(A)_{/Z}.
\]

\begin{prop}
The automorphism group $\Aut(\mathbf{B}_{2}(A))$ is generated by
\[
\{ \zeta_*, \delta_*, \pi, s_*, \tau_*\},
\]
and the complete set of defining relators are as follows.
\begin{itemize}
\item $\delta_*^2$;
\item $\zeta_*\pi=\pi\zeta_*\delta_*$, $(\pi\delta_*)^2$, $[s_*, \zeta_*]$, $[s_*, \delta_*]$, $(\tau_*\zeta_*)^2$, $(\tau_*\delta_*)^2$;
\item $s_*^2$, $\pi^2$, $\tau_*^2$, $(\tau_*s_*)^2$, $(\tau_*\pi)^2=\delta_*$.
\end{itemize}
\end{prop}
\begin{proof}
This follows by the direct computation using the presentations for inner and outer automorphism groups above and the obvious short exact sequence
\[
1\to\Inn(\mathbf{B}_{2}(A))\to\Aut(\mathbf{B}_{2}(A))\to\Out(\mathbf{B}_{2}(A))\to1.
\]
\end{proof}

Recall that $\tv(\mathbf{B}_{2}(A))$ is an infinite dihedral group admitting the group presentation
\[
\tv(\mathbf{B}_{2}(A))=\left\langle \bar\pi, s_*|s_*^2,\bar\pi^2\right\rangle,
\]
where $\bar\pi$ corresponds to $\pi s_*\pi$ in $\Aut(\mathbf{B}_2(A))$. Hence
\[
\Out^*(\mathbf{B}_2(A))=\Out(\mathbf{B}_2(A))/\tv(\mathbf{B}_2(A))=\langle \tau_*|\tau_*^2\rangle\times\langle\pi|\pi^2\rangle,
\]
and there is a commutative diagram whose rows are exact and split as follows.
\[
\xymatrix{
1\ar[r]& \Inn(\mathbf{B}_2(A))\ar[r]\ar[d]_{\simeq}&\mathcal{M}^*_2(A)_{/Z}\ar[r]\ar@{^(->}[d]^f&\langle\tau_*|\tau_*^2\rangle\ar[r]\ar@{^(->}[d]^{\bar f}&1\\
1\ar[r]& \Inn(\mathbf{B}_2(A))\ar[r]&\Aut^*(\mathbf{B}_2(A))\ar[r]&\Out^*(\mathbf{B}_2(A))\ar[r]&1,
}
\]
where the upper row comes from the Birman exact sequence. Since
\[
\coker f \simeq \coker \bar f = \langle \pi|\pi^2\rangle\subset\Aut^*(\mathbf{B}_2(A)),
\]
we get
\[
\Aut^*(\mathbf{B}_2(A))=\mathcal{M}^*_2(A)_{/Z}\rtimes\langle\pi|\pi^2\rangle.
\]

Finally, this implies that
\begin{align*}
\Aut(\mathbf{B}_2(A))&=\tv(\mathbf{B}_2(A))\rtimes\Aut^*(\mathbf{B}_2(A))\\
&= \tv(\mathbf{B}_2(A))\rtimes\mathcal{M}^*_2(A)_{/Z}\rtimes\langle\pi|\pi^2\rangle.
\end{align*}

This is the only case in which $\Aut(\mathbf{B}_n(\Sigma))$ is not isomorphic to $\tv(\mathbf{B}_n(\Sigma))\rtimes\mathcal{M}^*_n(\Sigma)_{/Z}$.
\begin{proof}[Proof of Theorem~\ref{thm:AutBn}]
Notice that there are no conditions on $\kappa$ or $\chi$.

All cases for $\kappa\ge 4$ and $n\ge 2$ have been covered already by Theorem~\ref{thm:AutBn} for generic cases and the discussion above for exceptional cases.

For $\kappa\le 3$, there are only 3 cases, which are $\mathbf{B}_2(\mathbb{R}^2), \mathbf{B}_2(S^2)$ and $\mathbf{B}_3(S^2)$.
Since the first two groups are abelian and so are the mapping class groups $\mathcal{M}^*_2(\mathbb{R}^2)$ and $\mathcal{M}^*_2(S^2)$, it is obvious that the above formula holds.

The group $\mathbf{B}_3(S^2)$ is finite of order 12 and $\mathbf{B}_3(S^2)_{/Z}$ is isomorphic to the symmetric group $\mathbf{S}_3$ on 3 letters whose center and outer automorphism groups are trivial. This implies that $\Aut(\mathbf{B}_3(S^2))$ is a direct product of $\tv(\mathbf{B}_3(S^2))$ and $\Inn(\mathbf{B}_3(S^2))$. Since $\Inn(\mathbf{B}_3(S^2))\simeq\mathcal{M}^*_3(S^2)_{/Z}$, this completes the proof.
\end{proof}

\end{document}